\documentclass[letterpaper,10pt]{article}

\usepackage[letterpaper,left=1in,right=1in,top=1.5in,bottom=1.5in]{geometry}

\newcommand{\myauthor}{Benjamin Antieau and Jeremiah Heller}

\newcommand{\mytitle}{Some remarks on topological K-theory of dg categories}

\title{Some remarks on topological $K$-theory of dg categories}

\author{Benjamin Antieau\footnote{Benjamin Antieau was supported
by NSF Grant DMS-1552766.}~ and Jeremiah Heller\footnote{Jeremiah Heller was supported by NSF Grant  DMS-1710966.}}

\date{\today}
 
\usepackage[pdfstartview=FitH,
            pdfauthor={\myauthor},
            pdftitle={\mytitle},
            colorlinks,
            linkcolor=reference,
            citecolor=citation,
            urlcolor=e-mail,
            backref]{hyperref}

\usepackage{amsmath}
\usepackage{amscd}
\usepackage{amsbsy}
\usepackage{amssymb}
\usepackage{verbatim}
\usepackage{eufrak}
\usepackage{eucal}
\usepackage{microtype}
\usepackage{mathrsfs}
\usepackage{amsthm}
\usepackage{xspace}
\usepackage[all,cmtip]{xy}
\usepackage{tikz}
\usetikzlibrary{matrix,arrows}
\usepackage[abbrev,lite,alphabetic]{amsrefs}
\usepackage{dsfont}

\usepackage{fancyhdr}
\usepackage{makeidx}
\makeindex

\usepackage{color}
\definecolor{todo}{rgb}{1,0,0}
\definecolor{conditional}{rgb}{0,1,0}
\definecolor{e-mail}{rgb}{0,.40,.80}
\definecolor{reference}{rgb}{.20,.60,.22}
\definecolor{mrnumber}{rgb}{.80,.40,0}
\definecolor{citation}{rgb}{0,.40,.80}
\definecolor{soldierblue}{RGB}{55,83,115}
\definecolor{aquamarine}{RGB}{102,205,170}

\let\oldmarginpar\marginpar
\renewcommand\marginpar[1]{\-\oldmarginpar[\raggedleft\footnotesize #1]%
{\raggedright\footnotesize #1}}

\newcommand{\Cscr}{\mathcal{C}}
\newcommand{\Dscr}{\mathcal{D}}

\newcommand{\Pscr}{\mathcal{P}}

\newcommand{\Sscr}{\mathcal{S}}
\newcommand{\Tscr}{\mathcal{T}}

\newcommand{\C}{\mathrm{C}}

\newcommand{\K}{\mathrm{K}}
\renewcommand{\L}{\mathrm{L}}

\renewcommand{\P}{\mathrm{P}}

\newcommand{\W}{\mathrm{W}}

\renewcommand{\AA}{\mathds{A}}
\newcommand{\CC}{\mathds{C}}

\newcommand{\PP}{\mathds{P}}

\newcommand{\sm}{\mathrm{sm}}
\newcommand{\aff}{\mathrm{aff}}
\renewcommand{\top}{\mathrm{top}}

\newcommand{\nc}{\mathrm{nc}}

\newcommand{\op}{\mathrm{op}}

\newcommand{\cn}{\mathrm{cn}}

\newcommand{\st}{\mathrm{st}}
\newcommand{\sst}{\mathrm{sst}}
\newcommand{\semi}{\mathrm{semi}}

\newcommand{\Top}{\Tscr\mathrm{op}}
\newcommand{\Sp}{\mathcal{S}\mathrm{p}}

\newcommand{\wkeq}{\simeq}

\renewcommand{\geq}{\geqslant}

\newcommand{\KGL}{\mathrm{KGL}}
\renewcommand{\Re}{\mathrm{Re}}

\newcommand{\Cat}{\mathrm{Cat}}

\newcommand{\KH}{\mathrm{KH}}

\newcommand{\ku}{\mathrm{ku}}
\newcommand{\KU}{\mathrm{KU}}

\DeclareMathOperator{\Hom}{Hom}
\newcommand{\Map}{\mathrm{Map}}

\DeclareMathOperator{\Shv}{Shv}

\newcommand{\Perf}{\mathrm{Perf}}
\DeclareMathOperator{\Sing}{Sing}

\newcommand{\Aff}{\mathrm{Aff}}
\newcommand{\Sch}{\mathrm{Sch}}
\newcommand{\Sm}{\mathrm{Sm}}

\DeclareMathOperator*{\colim}{colim}

\newcommand{\Nis}{\mathrm{Nis}}

\DeclareMathOperator{\Spec}{Spec}

\newcommand{\we}{\simeq}
\newcommand{\iso}{\cong}

\theoremstyle{plain}
\newtheorem{theorem}{Theorem}[section]
\newtheorem*{theorem*}{Theorem}
\newtheorem{lemma}[theorem]{Lemma}

\newtheorem{proposition}[theorem]{Proposition}

\newtheoremstyle{named}{}{}{\itshape}{}{\bfseries}{.}{.5em}{#1 \thmnote{#3}}
\theoremstyle{named}

\theoremstyle{definition}
\newtheorem{definition}[theorem]{Definition}

\newtheorem{question}[theorem]{Question}

\newtheorem{remark}[theorem]{Remark}

\newcommand{\ul}[1]{\underline{\smash{#1}}}

\begin{document}

\maketitle

\begin{abstract}
    \noindent
    Using techniques from motivic homotopy theory, we prove a
    conjecture of Anthony Blanc about semi-topological $K$-theory of dg categories
    with finite coefficients. Along the way, we show that the connective semi-topological
    $K$-theories defined by Friedlander-Walker and by Blanc agree for
    quasi-projective complex varieties and we study \'etale descent of
    topological $K$-theory of dg categories.


    \paragraph{Mathematics Subject Classification 2010.}
    Primary:
    \href{http://www.ams.org/mathscinet/msc/msc2010.html?t=14Fxx&btn=Current}{14F42},
    \href{http://www.ams.org/mathscinet/msc/msc2010.html?t=19Dxx&btn=Current}{19D55},
    \href{http://www.ams.org/mathscinet/msc/msc2010.html?t=19Exx&btn=Current}{19E08}.\\
    Secondary:
    \href{http://www.ams.org/mathscinet/msc/msc2010.html?t=14Fxx&btn=Current}{14F20},
    \href{http://www.ams.org/mathscinet/msc/msc2010.html?t=16Exx&btn=Current}{16E45},
    \href{http://www.ams.org/mathscinet/msc/msc2010.html?t=55Nxx&btn=Current}{55N15}.
\end{abstract}

\section{Introduction}\label{sec:intro}

Blanc defines~\cite{blanc} semi-topological and
topological $K$-theory functors
$$\K^{\st},\K^\top:\Cat_\CC\rightarrow\Sp,$$
where $\Cat_\CC$ denotes the $\infty$-category of small idempotent complete
pretriangulated dg categories over $\CC$\footnote{By work of L. Cohn~\cite{cohn}, $\Cat_\CC$ is
equivalent to the $\infty$-category of small idempotent complete $\CC$-linear
stable $\infty$-categories.} ($\CC$-linear dg categories in this paper for short) and $\Sp$ is the $\infty$-category of
spectra. When $\Cscr$ is a $\CC$-linear dg category, there are natural maps
$\K(\Cscr)\rightarrow\K^{\st}(\Cscr)\rightarrow\K^\top(\Cscr)$. Moreover,
$\K^\st(\Cscr)$ is a $\ku$-module spectrum, and, by definition,
$$\K^\top(\Cscr)\we\K^\st(\Cscr)[\beta^{-1}]\we\K^\st(\Cscr)\otimes_{\ku}\KU,$$
where $\beta\in\pi_2\ku$ is the Bott element.

Let $\Sch_\CC$ denote the category of separated $\CC$-schemes of finite type.
If $F:\Cat_\CC\rightarrow\Sp$ is a functor and $\Cscr$ is a $\CC$-linear
dg category,
then we define a presheaf
$\ul{F}(\Cscr):\Sch_\CC^\op\rightarrow\Sp$ by the formula
$\ul{F}(\Cscr)(X)\we F(\Perf_X\otimes_\CC\Cscr)$.
In other words, $\ul{F}(\Cscr)$ is the composition
of the functor
$\Perf:\Sch^\op_\CC\rightarrow\Cat_\CC$, the endofunctor
$$\Cat_\CC\rightarrow\Cat_\CC\qquad\Dscr\mapsto\Dscr\otimes_\CC\Cscr,$$
and the functor $F:\Cat_\CC\rightarrow\Sp$.
In many cases, we will use the restriction of
$\ul{F}(\Cscr)$ to $\Aff_\CC^\op$, $\Sm_\CC^\op$, or
$\Aff_\CC^{\sm,\op}\we\Sm_\CC^{\aff,\op}$
the opposites of the categories of affine $\CC$-schemes of finite type,
smooth separated $\CC$-schemes of finite type, and smooth affine $\CC$-schemes, respectively.

In this paper, we prove three theorems about semi-topological and
topological $K$-theory of dg categories. First, we prove that
$\K^\sst(X)\we\K^{\cn,\st}(\Perf_X)$ when $X$ is a quasi-projective complex variety,
where $\K^\sst(X)$ is the semi-topological $K$-theory spectrum defined by
Friedlander and Walker in~\cite{friedlander-walker-comparing} and
$\K^{\cn,\st}(\Perf_X)$ is the connective version of Blanc's semi-topological
$K$-theory. Second, we
prove a conjecture of Blanc, stating that $\K(\Cscr)/n\we\K^\st(\Cscr)/n$ for
$n\geq 1$ and any $\CC$-linear dg category $\Cscr$. Third, we prove that
$\ul{\K}^\top(\Cscr)$ is $\AA^1$-invariant and a hypersheaf for the \'etale
topology on $\Sm_\CC$ for any $\CC$-linear dg category $\Cscr$.
Put together, the last two theorems imply that $\ul{\K}(\Cscr)/n$ satisfies
\'etale hyperdescent after inverting the Bott element.

The first theorem has
also been obtained by Blanc and Horel and they also made progress toward the
second theorem along the same lines as the argument we give.

\paragraph{Acknowledgments.} BA would like to thank Tasos Moulinos for
patiently explaining Blanc's work on semi-topological $K$-theory to him on
several occasions. Both authors express their gratitude to Anthony Blanc for
his email comments on this topic and for looking over a preliminary version of
the paper. They also are grateful to a careful referee who made several nice
suggestions for improvements. Finally, this paper would probably not be
possible without the dogs Boschko and Lima, who created the opportunity for the
authors to work together.

\section{Comparison of semi-topological $K$-theories}\label{sec:comparison}

The original definition of semi-topological $K$-theory is for complex varieties
and goes back to work of Friedlander and
Walker~\cites{friedlander-walker,friedlander-walker-comparing}.
They construct spectra $\K^\semi(X)$ and $\K^\sst(X)$ when $X$ is
quasi-projective and they give a natural map $\K^\sst(X)\rightarrow\K^\semi(X)$. When $X$ is projective and
weakly normal, this map is an equivalence
by~\cite{friedlander-walker-comparing}*{Theorem~1.4}. In their
survey~\cite{friedlander-walker-handbook}, they settle on $\K^\sst(X)$ as the
`correct' definition of semi-topological $K$-theory of quasi-projective complex
varieties. It is natural to wonder about the
relationship between $\K^{\sst}(X)$ and $\K^{\st}(\Perf_X)$. We prove that
they are in fact equivalent. Blanc has communicated to us that he was
aware of this fact, although it was open at the time of~\cite{blanc}.

We recall the definition of semi-topological $K$-theory of dg categories
from~\cite{blanc}. Let $\Sch_\CC\rightarrow\Pscr_{\Sp}(\Sch_\CC)$ be the
spectral Yoneda functor, where $\Pscr_{\Sp}(\Sch_\CC)$ is the stable
presentable $\infty$-category
of presheaves of spectra on $\Sch_\CC$. Let
$\Sch_\CC\rightarrow\Sp$ be the composition of $$\Sch_\CC\rightarrow\Sscr\qquad
X\mapsto \Sing{X(\CC)}$$ with the suspension spectrum functor
$\Sigma^\infty_+:\Sscr\rightarrow\Sp$, where $\Sscr$ denotes the
$\infty$-category of topological spaces. Define the topological realization
$\Re:\Pscr_{\Sp}(\Sch_\CC)\rightarrow\Sp$ as the left Kan extension
\begin{equation}\label{eq:2}
    \xymatrix{
\Sch_\CC\ar[rr]^{X\mapsto\Sigma^\infty \Sing{X(\CC)}_+}\ar[d] &&   \Sp\\
\Pscr_{\Sp}(\Sch_\CC).\ar@{.>}[urr]^{\rm Re}&&
}\end{equation} Given $\Cscr\in\Cat_\CC$, there is the presheaf $\ul{\K}(\Cscr)$ as defined
in Section~\ref{sec:intro}, where $\K:\Cat_\CC\rightarrow\Sp$ denotes
nonconnective $K$-theory as defined for example in~\cite{bgt1}.

\begin{definition}[Blanc~\cite{blanc}]
    The semi-topological $K$-theory of
    $\Cscr$ is the spectrum $\K^\st(\Cscr)=\Re(\ul{\K}(\Cscr))$.
    More generally, let $f:X\rightarrow\Spec\CC$ be a separated $\CC$-scheme of
    finite type and let $\Sch_X$ be the category of separated $X$-schemes of
    finite presentation. There is an adjunction
    $f^*:\Pscr_{\Sp}(\Sch_\CC)\rightleftarrows\Pscr_{\Sp}(\Sch_X):f_*$ defined
    in the usual way. We let $\ul{\K}^\st(\Cscr):\Sch^\op\rightarrow\Sp$ be the presheaf
    with value at $f:X\rightarrow\CC$ given by
    $\ul{\K}^\st(\Cscr)(X)=\Re(f_*f^*\ul{\K}(\Cscr))$. In particular,
    $\ul{\K}^\st(\Cscr)(X)\we\K^\st(\Perf_X\otimes_\CC\Cscr)$ and
    $\ul{\K}^\st(\Cscr)(\Spec\CC)\we\K^\st(\Cscr)$.
\end{definition}

\begin{definition}[Blanc~\cite{blanc}]
    If we apply the same construction with connective $K$-theory $\K^\cn$ we
    obtain a connective version of semi-topological $K$-theory, namely
    $\K^{\cn,\st}(\Cscr)=\Re(\ul{\K}^{\cn}(\Cscr))$, where $\ul{\K}^\cn(\Cscr)$
    is the presheaf of connective spectra
    $\ul{\K}^\cn(\Cscr)(X)\we\K^{\cn}(\Perf_X\otimes_{\CC}\Cscr)$.
    This is the theory denoted by $\ul{\tilde{\K}}(\Cscr)$ in Blanc's paper.
\end{definition}

The following theorem has also been obtained by Blanc and Geoffroy Horel (private communication).

\begin{theorem}\label{thm:comparison}
    If $X$ is a quasi-projective complex variety, then there is a natural equivalence
    $\K^{\sst}(X)\we\K^{\cn,\st}(\Perf_X)$.
\end{theorem}

\begin{proof}
    To begin, we give the definition of $\K^{\sst}(X)$
    after~\cite{friedlander-walker-comparing}*{Definition~1.1}. Let
    $\widetilde{\Top}$ be a small category of topological spaces and continuous
    maps containing at least the essential image of
    $r:\Sch_\CC\xrightarrow{U\mapsto U(\CC)}\Top$ and the topological simplices
    $\Delta^n_\top$. Friedlander and Walker consider the left Kan extension
    $$\xymatrix{
        \Sch_\CC^{\op}\ar[rr]^{\ul{\K}^\cn(X)}\ar[d]_{U\mapsto U(\CC)}    &&   \Sp\\
        \widetilde{\Top}^{\op}\ar[urr]_{r^*\ul{\K}^\cn(X)}.&&
    }$$
    By definition, if $Y$ is a topological space, then
    $$r^*\ul{\K}^\cn(X)(Y)\we\colim_{Y\rightarrow
    U(\CC)}\ul{\K}^\cn(X)(U)\we\colim_{Y\rightarrow
    U(\CC)}\ul{\K}^\cn(X\times_\CC U).$$ Evaluating $r^*\ul{\K}^\cn(X)$ at the
    cosimplicial space $\Delta^\bullet_\top$, we obtain a simplicial spectrum
    $r^*\ul{\K}^\cn(X)(\Delta^\bullet_\top)$. The semi-topological $K$-theory of
    $Y$ is defined to be
    $$\K^{\sst}(X)=|\ul{\K}^\cn(X)(\Delta^\bullet_\top)|.$$
    Note that this process is precisely the composition of the functors
    $$
    \widetilde{\Re}_{\Sp}:\Pscr_{\Sp}(\Sch_\CC)\xrightarrow{r^*}\Pscr_{\Sp}(\widetilde{\Top})\xrightarrow{s_*}\Pscr_{\Sp}(\Delta)\rightarrow\Sp
    $$
    applied to $\ul{\K}^\cn(X)$, where $s$ denotes the inclusion of $\Delta$ into
    $\widetilde{\Top}$ classifying the cosimplicial space
    $\Delta^\bullet_{\top}$ and the final arrow is geometric realization of a
    simplicial spectrum. This composition is the stabilization of
    \begin{equation*}
        \widetilde{\Re}:\Pscr(\Sch_\CC)\xrightarrow{r^*}\Pscr(\widetilde{\Top})\xrightarrow{s_*}\Pscr(\Delta)\xrightarrow{|-|}\Sscr,.
    \end{equation*}
    where $\Pscr(\Delta)$ is the $\infty$-category of presheaves of spaces on
    the simplex category $\Delta$, or in other words the $\infty$-category of
    simplicial spaces, and $|-|$ denotes geometric realization.
    To prove the theorem it suffices to prove that
    $\widetilde{\Re}:\Pscr(\Sch_\CC)\rightarrow\Sscr$ is equivalent to the
    functor $\Pscr(\Sch_\CC)\rightarrow\Sscr$ obtained via the unstable version
    of the left Kan extension in~\eqref{eq:2}:
    $$\xymatrix{
    \Sch_\CC\ar[rr]^{U\mapsto\Sing{U(\CC)}}\ar[d]&&\Sscr\\
    \Pscr(\Sch_\CC).\ar[rru]^{\Re}&&
    }$$
    To prove that $\widetilde{\Re}\we\Re$, note first that both
    functors are left adjoints because $\Pscr(\Sch_\CC)$ is presentable. The
    only thing to check is that $s_*$ preserves colimits, but this follows
    because colimits in presheaf categories are computed pointwise (see for
    example~\cite{htt}*{Corollary~5.1.2.3}). Thus, it suffices to prove that
    the restrictions of $\widetilde{\Re}$ and $\Re$ to $\Sch_\CC$ are equivalent.
    On the one hand, we know that $\Re(U)\we U(\CC)$ in $\Sscr$ for
    $U\in\Sch_\CC$. On the
    other hand, by definition,
    $r^*U(\Delta^n_\top)\we\colim_{\Delta^n_\top\rightarrow V(\CC)}\Hom_{\Sch_\CC}(V,U)$.
    Using Jouanolou's device~\cite{jouanolou}, we see that any map $\Delta^n_\top\rightarrow
    V(\CC)$ factors through $W(\CC)\rightarrow V(\CC)$ where $W(\CC)$ is affine
    and $\W(\CC)\rightarrow V(\CC)$ is a vector bundle torsor. Thus, by~\cite{friedlander-walker}*{Proposition~4.2},
    $r^*U(\Delta^n_\top)\iso\Hom_{\Top}(\Delta^n_{\top},U(\CC))\iso\Sing{U(\CC)}_n$
    and hence $\widetilde{\Re}(U)\we\Sing{U(\CC)}$, as desired.
\end{proof}

\begin{remark}
    A theorem of Friedlander and Walker says that when $X$ is smooth and quasi-projective, 
    $\K^\sst(X)[\beta^{-1}]\we\KU(X(\CC))$, the complex $K$-theory spectrum of
    the space of $\CC$-points of $X$
    (see~\cite{friedlander-walker-handbook}*{Theorem~32}). It follows from the
    theorem that
    $\K^\top(\Perf_X)=\K^{\st}(\Perf_X)[\beta^{-1}]\we\K^{\sst}(X)[\beta^{-1}]\we\KU(X(\CC))$
    when $X$ is smooth and quasi-projective, where
    $\K^{\st}(\Perf_X)\we\K^{\cn,\st}(\Perf_X)$ because $X$ is smooth and
    by~\cite{blanc}*{Theorem~3.18}. This gives a new proof of one of
    the main theorems of Blanc's paper~\cite{blanc}*{Theorem~1.1(b)} in the
    special case of $X$ smooth and quasi-projective. Blanc's theorem says more
    generally that if $X$ is separated and finite type over $\CC$, then
    $\K^\top(\Perf_X)\we\KU(X(\CC))$.
\end{remark}

\begin{remark}
    It is clear that one could have defined a nonconnective version
    $\K^{\nc,\sst}(X)$ of Friedlander and Walker's
    $\K^{\sst}(X)$ simply by replacing connective $K$-theory with nonconnective
    $K$-theory in~\cite{friedlander-walker-comparing}*{Definition~1.1}. If this is done, then the proof of
    Theorem~\ref{thm:comparison} goes through and shows that there are natural
    equivalences $\K^{\nc,\sst}(X)\we\K^{\st}(X)$ for quasi-projective
    complex varieties $X$.
\end{remark}

\section{Blanc's conjecture}\label{sec:conjecture}

Let $\C\subseteq {\rm Sch}_\CC$ be a full subcategory closed under taking
products with $\AA^1_\CC$. Let $\Pscr_{\Sp}^{\AA^1}(\C)\subseteq\Pscr_{\Sp}(\C)$ be the full subcategory
of $\AA^1$-invariant presheaves of spectra, i.e., those $F$ such that the
pullback maps $F(X)\rightarrow F(X\times_\CC\AA^1)$ are equivalences for all $X\in\C$.
The inclusion has a left adjoint,
$\L_{\AA^1}:\Pscr_{\Sp}(\C)\rightarrow\Pscr_{\Sp}^{\AA^1}(\C)$. A map $F\rightarrow G$ is an
$\AA^1$-equivalence if $\L_{\AA^1}F\rightarrow\L_{\AA^1}G$ is an equivalence.
Given a presheaf of spectra $F$ on $\C$, we let $\Sing^{\AA^1}F$ be the presheaf defined by
$$
(\Sing^{\AA^1}F)(X) = \colim_{\Delta}F(X\times_\CC\Delta^{\bullet}_{\CC}),
$$ 
where $\Delta^\bullet_\CC$ is the standard cosimplicial affine scheme. It is a
standard fact that $\Sing^{\AA^1}F$ is $\AA^1$-invariant in the sense that for
every $X\in\C$, the
pullback map $(\Sing^{\AA^1}F)(X)\rightarrow(\Sing^{\AA^1}F)(X\times_\CC\AA^1)$ induced
by the projection $X\times_\CC\AA^1_\CC\rightarrow X$ is an equivalence.
Moreover, $F\rightarrow\Sing^{\AA^1}F$ is an $\AA^1$-equivalence. It follows
that $\Sing^{\AA^1}F\we\L_{\AA^1}F$ for all $F$ and that
if $F$ is $\AA^1$-invariant, then the natural transformation
$F\rightarrow\L_{\AA^1}F$ is an equivalence. For proofs of these facts, 
see~\cite{morel-voevodsky}*{Section~2.3}.
From the natural equivalences $\L_{\AA^1}F\we\Sing^{\AA^1}F$, we see that if $i:\C'\subseteq \C$ is a subcategory
(also closed under taking products with $\AA^1_\CC$) and if $F$ is a presheaf
on $\C$, then $i^*\L_{\AA^1}F =\L_{\AA^1}i^*F$.

Let $\Cscr$ be a $\CC$-linear dg category. Then,
$\L_{\AA^1}\ul{\K}(\Cscr)\we\ul{\KH}(\Cscr)$, where
$\KH:\Cat_\CC\rightarrow\Sp$ is the homotopy $K$-theory of dg categories, as
defined for example in~\cite{tabuada-kh}. If $F:\Sch^\op\to\Sp$ is a presheaf of
spectra, write $F^\st$ for the presheaf with value at $f:X\rightarrow\Spec\CC$
in $\Sch_\CC$ the spectrum $F^\st(X)\we\Re(f_*f^*F)$.

\begin{lemma}
    If $F:\Sch^{\op}\to \Sp$ is a presheaf of spectra, then
    $$
    F^\st\we (\L_{\AA^1}F)^\st\we\L_{\AA^1}(F^\st).
    $$ In particular, if $\Cscr$ is a $\CC$-linear dg category, then 
    $$
    \K^{\st}(\Cscr) \wkeq \KH^{\st}(\Cscr),
    $$
    where $\KH^\st(\Cscr)=\Re(\ul{\KH}(\Cscr))$.
\end{lemma}

\begin{proof}
    Since $\AA^1_\CC\rightarrow\Spec\CC$ realizes to an equivalence in $\Sp$,
    the realization functor $\Re:\Pscr_{\Sp}(\Sch_\CC)\rightarrow\Sp$ factors
    through the $\AA^1$-localization
    $\Pscr_{\Sp}(\Sch_\CC)\rightarrow\Pscr_{\Sp}^{\AA^1}(\Sch_\CC)$, which is
    modeled concretely by $\L_{\AA^1}$. This proves that
    $F^\st\we(\L_{\AA^1}F)^\st$. If we prove that $(\L_{\AA^1}F)^\st$ is
    $\AA^1$-invariant, then we will have proved that $\L_{\AA^1}(F^\st)\we
    F^\st$. It is enough to prove that
    $\st:\Pscr_{\Sp}(\Sch_\CC)\rightarrow\Pscr_{\Sp}(\Sch_\CC)$ preserves
    $\AA^1$-invariant presheaves. Let $G$ be $\AA^1$-invariant. If
    $f:X\rightarrow\Spec\CC$, then $f_*f^*G\we g_*g^*G$
    where $g:X\times_\CC\AA^1_\CC\rightarrow\Spec\CC$ since $G$ is
    $\AA^1$-invariant. Thus,
    $G^\st(X)\we\Re(f_*f^*G)\we\Re(g_*g^*G)\we G^\st(X\times_\CC\AA^1)$, as
    desired. The second claim follows from the equivalence
    $\ul{\K}^\st(\Cscr)\we\ul{\KH}^\st(\Cscr)$ of presheaves evaluated at $\Spec\CC$.
\end{proof}

Write $\Sp(\CC)$ for the $\infty$-category of motivic $\PP^1$-spectra over
$\CC$, $\ul{\Map}_{\Sp(\CC)}(-,-)$ for the internal mapping object,
and $\Map_{\Sp(\CC)}(-,-)$ for the (classical) mapping spectrum.  
A good reference for $\Sp(\CC)$ in the language of $\infty$-categories is~\cite{robalo}.

\begin{proposition}\label{prop:KGL}
    Let $\Cscr$ be a $\CC$-linear dg category.	
    There is a motivic spectrum $\KGL(\Cscr)\in \Sp(\CC)$ such that 
    $$
    \Map_{\Sp(\CC)}(\Sigma^{\infty}_{\PP^1}X_+,\KGL(\Cscr)) \wkeq \ul{\KH}(\Cscr)(X)
    $$
    for any $X\in \Sm_\CC$.
\end{proposition}

\begin{proof}
	Below, we use the (nonstandard) notation 
	$T = (\PP^1,\infty)$ for the based scheme and  for the duration of this
    proof write $\PP^1$ only to denote the unbased scheme.

	Write $\Sp_{S^1}(\CC)$ for the $\infty$-category of $\AA^1$-invariant, Nisnevich sheaves of spectra on $\Sm_\CC$. 
	Note that  $\ul{\KH}(\Cscr)\in \Sp_{S^1}(\CC)$. Indeed, it is
    $\AA^1$-invariant by definition and it is a Nisnevich sheaf, because it is the
    restriction of a localizing invariant to $\Sm_\CC$
    (see~\cite{blanc}*{Theorem~1.1(c)}).
	
	By~\cite[Corollary 2.22]{robalo} for example, we have an equivalence
	$$
	\Sp(\CC) \wkeq 
	{\rm Stab}_{T}(\Sp_{S^1}(\CC)) := \lim (\Sp_{S^1}(\CC)
    \xleftarrow{\Omega_{T}} \Sp_{S^1}(\CC)\leftarrow \cdots ).
	$$
Let $\beta \in \pi_{0}\Map_{\Sp_{S^1}(\CC)}(T, \ul{\KH}(\Perf_\CC))\cong \KH_0((\PP^1,\infty))$ be the usual Bott element. Write as well $\beta:\ul{\KH}(\Cscr) \to \Map_{\Sp_{S^1}(\CC)}(T,\ul{\KH}(\Cscr))$ for the ``multiplication by $\beta$'' map, obtained from the  $\ul{\KH}(\Perf_\CC))$-module structure on  
 $\ul{\KH}(\Cscr)$. Now define $\KGL(\Cscr)\in \Sp(\CC)$ to be the ``constant'' spectrum 
    whose value is $\ul{\KH}(\Cscr)$ and structure maps 
    $\beta:\ul{\KH}(\Cscr) \to \Map_{\Sp_{S^1}(\CC)}(T,\ul{\KH}(\Cscr))$.

Since  $\ul{\KH}(\Cscr)$ is a localizing invariant and $\Perf_{\PP^1_X}\wkeq
\Perf_X \oplus \Perf_X$, the projective bundle formula holds in $\ul{\KH}(\Cscr)$:
    $$
    \ul{\KH}(\Cscr)(\PP^1_X)\we\ul{\KH}(\Cscr)(X)\oplus\ul{\KH}(\Cscr)(X)
    $$
    for $X\in\Sch_\CC$. This splitting identifies $\Map_{\Sp_{S^1}(\CC)}(T,\ul{\KH}(\Cscr))$ with $\ul{\KH}(\Cscr)$ via the map $\beta$ defined above. In particular,  we see that $\KGL(\Cscr)$ is a periodic motivic spectrum and $\Omega^{\infty}_{T}(\KGL(\Cscr))\wkeq \ul{\KH}(\Cscr)$.
	It is now immediate that 
    $$
    \Map_{\Sp(\CC)}(\Sigma^{\infty}_{T}X_+,\KGL(\Cscr)) \wkeq 
    \Map_{\Sp_{S^1}(\CC)}(\Sigma^{\infty}_{S^1}X_+,\ul{\KH}(\Cscr)) \wkeq
     \ul{\KH}(\Cscr)(X),
    $$
    for any $X\in \Sm_{\CC}$.

\end{proof}

Now, we prove Blanc's conjecture. Blanc has told us that Horel was exploring a
similar argument.

\begin{theorem}\label{thm:blancsconjecture}
    If $\Cscr$ is a $\CC$-linear dg category, then the natural map
    $\K(\Cscr)/n\rightarrow\K^{\st}(\Cscr)/n$ is an equivalence for any $n\geq 1$.
\end{theorem}

\begin{proof}
    By~\cite{blanc}*{Theorem~3.18}, we may compute the semi-topological $K$-theory of
    $\Cscr$ using only smooth $\CC$-schemes $\Sm_\CC\subseteq\Aff_\CC$. In
    fact, if $F:\Sch_\CC^\op\rightarrow\Sp$ is any presheaf of spectra, we can
    compute $\Re(F)$ by first restricting $F$ to $\Sm_\CC$ and then using the
    realization $\Re:\Pscr_{\Sp}(\Sm_\CC)\rightarrow\Sp$ given as the left Kan
    extension
    $$\xymatrix{
    \Sm_\CC\ar[rr]^{X\mapsto\Sigma^\infty\Sing{X(\CC)}_+}\ar[d] &&   \Sp\\
    \Pscr_{\Sp}(\Sm_\CC)\ar[urr]^\Re.&&
    }$$
    Let $\mathrm{E}$ 
    denote the constant presheaf on $\Sm_\CC$
    with value $\K(\Cscr)/n$. There is a natural map
    $\mathrm{E}\rightarrow\ul{\K}(\Cscr)/n$. The topological
    realization of $\mathrm{E}$ is $\K(\Cscr)/n$ since it is a
    constant sheaf. As topological realization factors through Nisnevich hypersheaves, it suffices to check that
    $\mathrm{E}\rightarrow\ul{\K}(\Cscr)/n$ induces an
    equivalence after Nisnevich hypersheafification. For this, it suffices to see
    that the natural map $\K(\Cscr)/n\rightarrow\K(R\otimes_\CC\Cscr)/n$ is an
    equivalence for every essentially smooth hensel local ring $R$
    over $\CC$.

    By the proposition above, $\ul{\KH}(\Cscr)$ is represented by a motivic
    spectrum, $\KGL(\Cscr)$. Thus, $\ul{\KH}(\Cscr)/n$ is represented by a
    motivic spectrum denoted $\KGL(\Cscr)/n$.
    Gabber-Suslin rigidity is valid for $\KGL(\Cscr)/n$
    by~\cite{hornbostel-yagunov}*{Corollary~0.4}.
    (As noted in loc.~cit., the normalization property of that result holds for
    any motivic spectrum over an algebraically closed field.) In particular,
    $\KH(R\otimes_\CC\Cscr)/n\rightarrow\KH(R/\mathfrak{m}\otimes_\CC\Cscr)/n\we\K(\Cscr)/n$
    is an equivalence
    for any essentially smooth hensel local ring $R$, where $\mathfrak{m}\subseteq R$ is the maximal
    ideal.  But by a result of
    Tabuada~\cite{tabuada-a1-modl}*{Theorem~1.2(i)}, whose proof essentially follows
    the argument of Weibel in the case of associative
    rings~\cite{weibel-kh}*{Proposition~1.6}, $\ul{\K}(\Cscr)/n$ is
    $\AA^1$-homotopy invariant so that $\ul{\K}(\Cscr)/n\we\ul{\KH}(\Cscr)/n$. It
    follows that
    $\K(R\otimes_\CC\Cscr)/n\rightarrow\K(R/\mathfrak{m}\otimes_\CC\Cscr)/n\we\K(\Cscr)/n$
    is an equivalence so that $\mathrm{E}\rightarrow\ul{\K}(\Cscr)/n$ is
    an equivalence, which is what we wanted to prove.
\end{proof}

\section{Descent for topological $K$-theory of dg catgories}

In this section we prove the following result.

\begin{theorem}\label{thm:descent}
    Let $\Cscr$ be a $\CC$-linear dg category.
    \begin{enumerate}
        \item[{\rm (i)}] The presheaf $\ul{\K}^{\top}(\Cscr):\Sm_\CC^\op\rightarrow\Sp$ satisfies \'etale hyperdescent.
        \item[{\rm (ii)}] The presheaf $\ul{\K}(\Cscr)/n[\beta^{-1}]:\Sm_\CC^\op\rightarrow\Sp$ satisfies \'etale hyperdescent.
    \end{enumerate}        
\end{theorem}

Part (ii) of the theorem is a noncommutative generalization of the main theorem
of Thomason~\cite{thomason}.
Indeed, if $\Cscr\we\Perf_X$ where $X$ is an essentially smooth separated $\CC$-scheme, then
$\ul{\K}(\Perf_X)/n[\beta^{-1}]$ is equivalent to the presheaf
$$Y\mapsto\K(Y\times_\CC X)/n[\beta^{-1}],$$
which satisfies \'etale hyperdescent by~\cite{thomason}*{Theorem~4.1}.
In general, we cannot improve the result to semi-topological $K$-theory.
Indeed, it is well-known that $\ul{\K}^\st(\Perf_X)/n\we\ul{\K}(\Perf_X)/n$
does not satisfy \'etale descent.
The Quillen--Lichtenbaum conjectures (which
follow from the, now proved, Bloch--Kato conjecture) give a bound on the
failure of \'etale hyperdescent for $K$-theory with finite coefficients.
For example, if $X$ is an essentially smooth separated $\CC$-scheme of Krull
dimension $d$, then
$$\K(X)/\ell\rightarrow\K(X)/\ell[\beta^{-1}]$$ is $2d$-coconnective.
(See~\cite{thomason-bott}*{Section~5} for a discussion of the bound $2d$.)
Recall that a map of spectra
$M\rightarrow N$ is $r$-coconnected if the induced map
$\pi_rM\rightarrow\pi_rN$ is an injection and $\pi_sM\rightarrow\pi_sN$ is an
isomorphism for $s>r$. Following the tradition of
proposing noncommutative versions of theorems known for $\CC$-linear dg categories
of the form $\Perf_X$, we ask the following question.

\begin{question}[Noncommutative Quillen-Lichtenbaum]
    If $\Cscr$ is a nice (probably smooth and proper) $\CC$-linear dg category,
    is $$\K(\Cscr)/n\rightarrow\K^{\top}(\Cscr)/n$$ is $r$-coconnective for some $r$.
\end{question}
 
To prove Theorem~\ref{thm:descent},
we make use of the topological realization functor ${\rm Re}:\Sp(\CC) \to \Sp$, extending the functor of taking complex points of a $\CC$-scheme. 
This functor factors through the
localization $\Pscr_{\Sp}(\Sm_\CC)\rightarrow\Shv_{\Sp}^{\Nis,\AA^1}(\Sm_\CC)$,
which is equivalent to $\Sp_{S^1}(\CC)$, the category of motivic $S^1$-spectra.
To see that it factors through the $T$-stabilization functor
$\Sp_{S^1}(\CC)\rightarrow\Sp(\CC)$, it is enough to note that the realization
of $T$ is $S^2$, which is already tensor invertible in $\Sp$. Thus, we have a
commutative diagram
$$\xymatrix{
\Sm_\CC\ar[r]^{X\mapsto\Sigma^\infty_+X(\CC)}\ar[d] &   \Sp&\\
\Pscr_{\Sp}(\Sm_\CC)\ar@{.>}[ur]^{\rm Re}\ar[r] &
\Sp_{S^1}(\CC)\ar@{.>}[u]^{\rm Re}\ar[r]&\Sp(\CC)\ar@{.>}[ul]_{\rm Re}
}$$
of realization functors, and we will abuse notation by not distinguishing them.

\begin{lemma}
    Let $\Cscr$ be a $\CC$-linear dg category. Then, there is an equivalence
    ${\rm Re}(\KGL(\Cscr))\we\K^\top(\Cscr)$.
\end{lemma}

\begin{proof}
    By definition, $\KGL(\Cscr)\in \lim (\Sp_{S^1}(\CC) \xleftarrow{\Omega_T} \Sp_{S^1}(\CC)\leftarrow \cdots)$ 
    is the periodic $T$-spectrum with value $\ul{\KH}(\Cscr)$ and structure
    maps induced by $\beta$:
    $$\ul{\KH}(\Cscr)\xrightarrow{\beta}\Omega_T\ul{\KH}(\Cscr)\xrightarrow{\beta}\Omega_T^2\ul{\KH}(\Cscr)\xrightarrow{\beta}\cdots.$$
    The realization functor $\Re:\Sp(\CC)\rightarrow\Sp$ factors through the
    equivalence $\Sp_{S^2}\Sp\we\Sp$, where $\Sp_{S^2}\Sp$ is the
    $\infty$-category of $S^2$-spectra in spectra. The realization functor
    sends $\KGL(\Cscr)$ to the $S^2$-spectrum
    $$\K^{\st}(\Cscr)\xrightarrow{\beta}\Omega^2\K^{\st}(\Cscr)\xrightarrow{\beta}\Omega^4\K^{\st}\xrightarrow{\beta}\cdots.$$
    The underlying spectrum of this $S^2$-spectrum in spectra is
    ``$\Omega^{2\infty}$'', the colimit of the diagram, which is by definition
    $\K^{\top}(\Cscr)$.
\end{proof}

\begin{lemma}
	If $X\in \Sm_{\CC}$, then
    ${\rm Re}(\ul{\Map}_{\Sp(\CC)}(\Sigma^{\infty}_{\P^1}X_+, \KGL(\Cscr))) \wkeq \ul{\K}^{\top}(\Cscr)(X)$.  
\end{lemma} 

\begin{proof}
    By adjunction,
    $$\ul{\Map}_{\Sp(\CC)}(\Sigma^\infty_{\PP^1}X_+,\KGL(\Cscr))\we\KGL(\Perf_X\otimes_\CC\Cscr).$$
    The claim follows from the fact that ${\rm
    Re}(\KGL(\Perf_X\otimes_\CC\Cscr))\we\K^\top(\Perf_X\otimes_\CC\Cscr)\we\ul{\K}^\top(\Cscr)(X)$.
  
\end{proof}

\begin{lemma}\label{lem:cohthy}
	If $X\in \Sm_{\CC}$, then 
	$\ul{\K}^{\top}(\Cscr)(X)\wkeq \Map_{\Sp}(\Sigma^{\infty}X(\CC)_+, \K^{\top}(\Cscr))$.
\end{lemma}

\begin{proof}
    Because the functor ${\rm Re}$ is symmetric monoidal, it
    commutes with internal mapping objects. Since $\Sigma^{\infty}_{\PP^1}X_+$ is dualizable
    in $\Sp(\CC)$, the statement of the lemma follows from the previous lemma.
\end{proof}

\begin{proof}[Proof of Theorem~\ref{thm:descent}]
    It follows from the equivalence $\ul{\K}(\Cscr)/n\wkeq
    \ul{\K}^{\st}(\Cscr)/n$ of Theorem~\ref{thm:blancsconjecture} that the second part follows from the first part.
    On the other hand, Lemma~\ref{lem:cohthy} shows that
    $\ul{\K}^{\top}(\Cscr)$ is the restriction of the cohomology theory on
    spaces represented by $\K^{\top}(\Cscr)$ to $\Sch_\CC$. It follows that it
    satisfies \'etale hyperdescent since any cohomology theory does (see for
    example~\cite{dugger-isaksen}*{Theorem~5.2}).
\end{proof}

%%%%%%%%%%%%%%%%%%%%%
%%% End material. %%%
%%%%%%%%%%%%%%%%%%%%%

% \vspace{20pt}
% \scriptsize
% \noindent
% Benjamin Antieau\\
% University of Illinois at Chicago\\
% Department of Mathematics, Statistics, and Computer Science\\
% 851 South Morgan Street, Chicago, IL 60607\\
% \texttt{antieau@math.uic.edu}
% 
% \vspace{10pt}
% \noindent
% David Gepner\\
% Purdue University\\
% Department of Mathematics\\
% 150 N. University Street, West Lafayette, IN 47907\\
% \texttt{dgepner@math.purdue.edu}


\addcontentsline{toc}{section}{References}



\begin{bibdiv}
\begin{biblist}

% \bib{abramovich-polishchuk}{article}{
%     author={Abramovich, Dan},
%     author={Polishchuk, Alexander},
%     title={Sheaves of $t$-structures and valuative criteria for
%     stable
%     complexes},
%     journal={J. Reine Angew. Math.},
%     volume={590},
%     date={2006},
%     pages={89--130},
%     issn={0075-4102},
% %     review={\MR{2208130}},
% %     doi={10.1515/CRELLE.2006.005},
% }

% \bib{abg}{article}{
%     author={Antieau, Benjamin},
%     author={Barthel, Tobias},
%     author={Gepner, David},
%     title={On localization sequences in the algebraic $K$-theory of ring spectra},
%     journal = {ArXiv e-prints},
%     eprint = {http://arxiv.org/abs/1412.4041},
%     year = {2014},
% }

% \bib{ag}{article}{
%     author={Antieau, Benjamin},
%     author={Gepner, David},
%     title={Brauer groups and \'etale cohomology in derived
%     algebraic
%     geometry},
%     journal={Geom. Topol.},
%     volume={18},
%     date={2014},
%     number={2},
%     pages={1149--1244},
%     issn={1465-3060},
% %     review={\MR{3190610}},
% %     doi={10.2140/gt.2014.18.1149},
% }

% \bib{agh}{article}{
%     author={Antieau, Benjamin},
%     author={Gepner, David},
%     author={Heller, Jeremiah},
%     title={On the theorem of the heart in negative $K$-theory},
%     journal = {ArXiv e-prints},
%     date={2016},
%     eprint={https://arxiv.org/abs/1610.07207},
% %     review={\MR{3190610}},
% %     doi={10.2140/gt.2014.18.1149},
% }

% % \bib{aw1}{article}{
% %     author={Antieau, Benjamin},
% %     author={Williams, Ben},
% %     title={The period-index problem for twisted topological
% %     $K$-theory},
% %     journal={Geom. Topol.},
% %     volume={18},
% %     date={2014},
% %     number={2},
% %     pages={1115--1148},
% %     issn={1465-3060},
% % %     review={\MR{3190609}},
% % %     doi={10.2140/gt.2014.18.1115},
% % }
% % 
% %  
% % % \bib{aw3}{article}{
% % %     author = {Antieau, Benjamin},
% % %     author = {Williams, Ben},
% % %     title = {The topological period-index problem for $6$-complexes},
% % %     journal = {J. Top.},
% % %     eprint = {http://doi:10.1112/jtopol/jtt042},
% % %     year = {2013},
% % % }
% % 
% % \bib{aw4}{article}{
% %     author={Antieau, Benjamin},
% %     author={Williams, Ben},
% %     title={Unramified division algebras do not always contain Azumaya
% %     maximal
% %     orders},
% %     journal={Invent. Math.},
% %     volume={197},
% %     date={2014},
% %     number={1},
% %     pages={47--56},
% %     issn={0020-9910},
% % %     review={\MR{3219514}},
% % %     doi={10.1007/s00222-013-0479-7},
% % }
% % 
% % \bib{aw7}{article}{
% %     author={Antieau, Benjamin},
% %     author={Williams, Ben},
% %     title={The prime divisors of the period and index of a Brauer
% %     class},
% %     journal={J. Pure Appl. Algebra},
% %     volume={219},
% %     date={2015},
% %     number={6},
% %     pages={2218--2224},
% %     issn={0022-4049},
% % %     review={\MR{3299728}},
% % %     doi={10.1016/j.jpaa.2014.07.032},
% % }
% % \bib{sga4-3}{book}{
% %     title={Th\'eorie des topos et cohomologie \'etale des sch\'emas. Tome 3},
% %     series={Lecture Notes in Mathematics, Vol. 305},
% %     note={S\'eminaire de G\'eom\'etrie Alg\'ebrique du Bois-Marie 1963--1964 (SGA 4);
% %     Dirig\'e par M. Artin, A. Grothendieck et J. L. Verdier. Avec la collaboration de P. Deligne et B. Saint-Donat},
% %     publisher={Springer-Verlag},
% %     place={Berlin},
% %     date={1973},
% %     pages={vi+640},
% % %     review={\MR{0354654 (50 \#7132)}},
% % }
% 
% \bib{asok-fasel-threefolds}{article}{
%     author={Asok, Aravind},
%     author={Fasel, Jean},
%     title={A cohomological classification of vector bundles on smooth affine threefolds},
%     journal={Duke Math. J.},
%     volume={163},
%     date={2014},
%     number={14},
%     pages={2561--2601},
%     issn={0012-7094},
% %     review={\MR{3273577}},
% %     doi={10.1215/00127094-2819299},
% }
% 
% \bib{asok-fasel-fourfolds}{article}{
%     author={Asok, Aravind},
%     author={Fasel, Jean},
%     title={Splitting vector bundles outside the stable range and $\AA^1$-homotopy sheaves of punctured affine spaces},
%     journal={to appear in J. AMS},
%     eprint={http://arxiv.org/abs/1209.5631},
% %     review={\MR{3273577}},
% %     doi={10.1215/00127094-2819299},
% }
% 
% % \bib{atiyah-segal}{article}{
% %     author={Atiyah, Michael},
% %     author={Segal, Graeme},
% %     title={Twisted $K$-theory},
% %     journal={Ukr. Mat. Visn.},
% %     volume={1},
% %     date={2004},
% %     number={3},
% %     pages={287--330},
% %     issn={1810-3200},
% %     translation={
% %     journal={Ukr. Math. Bull.},
% %     volume={1},
% %     date={2004},
% %     number={3},
% %     pages={291--334},
% %     issn={1812-3309},
% %     },
% % %     review={\MR{2172633 (2006m:55017)}},
% % }
% 
% % \bib{atiyah-segal-cohomology}{article}{
% %     author={Atiyah, Michael},
% %     author={Segal, Graeme},
% %     title={Twisted $K$-theory and cohomology},
% %     conference={
% %     title={Inspired by S. S. Chern},
% %     },
% %     book={
% %     series={Nankai Tracts Math.},
% %     volume={11},
% %     publisher={World Sci. Publ.,
% %     Hackensack, NJ},
% %     },
% %     date={2006},
% %     pages={5--43},
% % %     review={\MR{2307274 (2008j:55003)}},
% % %     doi={10.1142/9789812772688_0002},
% % }
% 
% % \bib{auel}{article}{
% %     author = {Auel, Asher},
% %     title = {Surjectivity of the total Clifford invariant and Brauer dimension},
% %     journal = {ArXiv e-prints},
% %     eprint = {http://arxiv.org/abs/1108.5728},
% %     year = {2011},
% % }
% % 
% % \bib{auel-remarks}{article}{
% %     author={Auel, Asher},
% %     title={Remarks on the Milnor conjecture over schemes},
% %     conference={
% %         title={Galois-Teichmüller Theory and Arithmetic Geometry (Kyoto, 2010)},
% %     },
% %     book={
% %         series={Advanced Studies in Pure Mathematics},
% %         volume={63},
% %         editor={Nakamura, H.},
% %         editor={Pop, F.},
% %         editor={Schneps, L.},
% %         editor={Tamagawa, A.},
% %     },
% %     year={2012},
% %     pages={1--30},
% % }
% % 
% % % \bib{auslander-goldman}{article}{
% % %     author={Auslander, Maurice},
% % %     author={Goldman, Oscar},
% % %     title={Maximal orders},
% % %     journal={Trans. Amer. Math. Soc.},
% % %     volume={97},
% % %     date={1960},
% % %     pages={1--24},
% % %     issn={0002-9947},
% % % %     review={\MR{0117252 (22 \#8034)}},
% % % }
% % 
% % \bib{auslander-goldman}{article}{
% %     author={Auslander, Maurice},
% %     author={Goldman, Oscar},
% %     title={The Brauer group of a commutative ring},
% %     journal={Trans. Amer. Math. Soc.},
% %     volume={97},
% %     date={1960},
% %     pages={367--409},
% %     issn={0002-9947},
% % %     review={\MR{0121392 (22 \#12130)}},
% % }

% \bib{bdr}{article}{
%     author={Baas, Nils A.},
%     author={Dundas, Bj\o rn Ian},
%     author={Rognes, John},
%     title={Two-vector bundles and forms of elliptic cohomology},
%     conference={
%         title={Topology,
%         geometry
%         and
%         quantum
%         field
%         theory},
%     },
%     book={
%         series={London
%         Math.
%         Soc.
%         Lecture
%         Note
%         Ser.},
%         volume={308},
%         publisher={Cambridge
%         Univ.
%         Press,
%         Cambridge},
%     },
%     date={2004},
%     pages={18--45},
% %     review={\MR{2079370}},
% %     doi={10.1017/CBO9780511526398.005},
% }
% 
% \bib{bdrr1}{article}{
%     author={Baas, Nils A.},
%     author={Dundas, Bj\o rn Ian},
%     author={Richter, Birgit},
%     author={Rognes, John},
%     title={Stable bundles over
%     rig categories},
%     journal={J. Topol.},
%     volume={4},
%     date={2011},
%     number={3},
%     pages={623--640},
%     issn={1753-8416},
%     % review={\MR{2832571}},
%     % doi={10.1112/jtopol/jtr016},
% }
% 
% \bib{bdrr2}{article}{
%     author={Baas, Nils A.},
%     author={Dundas, Bj\o rn Ian},
%     author={Richter, Birgit},
%     author={Rognes, John},
%     title={Ring completion of rig categories},
%     journal={J. Reine Angew. Math.},
%     volume={674},
%     date={2013},
%     pages={43--80},
% %     issn={0075-4102},
% %     review={\MR{3010546}},
% }

% \bib{balmer-schlichting}{article}{
%     author={Balmer, Paul},
%     author={Schlichting, Marco},
%     title={Idempotent completion of triangulated categories},
%     journal={J. Algebra},
%     volume={236},
%     date={2001},
%     number={2},
%     pages={819--834},
%     issn={0021-8693},
% %     review={\MR{1813503}},
% %     doi={10.1006/jabr.2000.8529},
% }
% 
% \bib{balmer-tabuada}{article}{
%     author={Balmer, Paul},
%     author={Tabuada, Gon{\c{c}}alo},
%     title={Fundamental isomorphism conjecture via non-commutative
%     motives},
%     journal={Math. Nachr.},
%     volume={286},
%     date={2013},
%     number={8-9},
%     pages={791--798},
%     issn={0025-584X},
% %     review={\MR{3066401}},
% %     doi={10.1002/mana.201200057},
% }

% % 
% % \bib{balmer-walter}{article}{
% %     author={Balmer, Paul},
% %     author={Walter, Charles},
% %     title={A Gersten-Witt spectral sequence for regular schemes},
% %     journal={Ann. Sci. \'Ecole Norm. Sup. (4)},
% %     volume={35},
% %     date={2002},
% %     number={1},
% %     pages={127--152},
% %     issn={0012-9593},
% % %     review={\MR{1886007 (2003d:19005)}},
% % %     doi={10.1016/S0012-9593(01)01084-9},
% % }
% % 

% \bib{barwick}{article}{
%     author={Barwick, Clark},
%     title={On exact $\infty$-categories and the Theorem of the Heart},
%     journal={Compos. Math.},
%     volume={151},
%     date={2015},
%     number={11},
%     pages={2160--2186},
%     issn={0010-437X},
% %     review={\MR{3427577}},
% %     doi={10.1112/S0010437X15007447},
% }
% 
% \bib{barwick-ktheory}{article}{
%     author={Barwick, Clark},
%     title={On the algebraic $K$-theory of higher categories},
%     journal={J. Topol.},
%     volume={9},
%     date={2016},
%     number={1},
%     pages={245--347},
%     issn={1753-8416},
% %     review={\MR{3465850}},
% %     doi={10.1112/jtopol/jtv042},
% }

% \bib{barwick-lawson}{article}{
%     author = {Barwick, Clark},
%     author = {Lawson, Tyler},
%     title = {Regularity of structured ring spectra and localization in K-theory},
%     journal = {ArXiv e-prints},
%     eprint = {http://arxiv.org/abs/1402.6038},
%     year = {2014},
% }
% 
% \bib{bass-problems}{article}{
%     author={Bass, Hyman},
%     title={Some problems in `classical'' algebraic $K$-theory},
%     conference={
%     title={Algebraic $K$-theory, II: `Classical'' algebraic
%     $K$-theory
%     and connections with arithmetic},
%     address={Proc. Conf., Battelle Memorial
%     Inst., Seattle, Wash.},
%     date={1972},
%     },
%     book={
%     publisher={Springer,
%     Berlin},
%     },
%     date={1973},
%     pages={3--73.
%     Lecture
%     Notes
%     in
%     Math.,
%     Vol.
%     342},
% %     review={\MR{0409606}},
% }

% % % \bib{baum-browder}{article}{
% % % 	title = {The cohomology of quotients of classical groups},
% % % 	volume = {3},
% % % 	journal = {Topology},
% % % 	author = {Baum, Paul F.},
% % %     author={Browder, William},
% % % 	year = {1965},
% % % 	pages = {305--336}
% % % }
% \bib{bbd}{article}{
%     author={Be{\u\i}linson, A. A.},
%     author={Bernstein, J.},
%     author={Deligne, P.},
%     title={Faisceaux pervers},
% %     language={French},
%     conference={
%     title={Analysis and topology on singular
%     spaces, I},
%     address={Luminy},
%     date={1981},
%     },
%     book={
%     series={Ast\'erisque},
%     volume={100},
%     publisher={Soc.
%     Math.
%     France,
%     Paris},
%     },
%     date={1982},
%     pages={5--171},
% %     review={\MR{751966 (86g:32015)}},
% }

% \bib{berman}{article}{
%     author={Berman, John},
%     title={Higher algebraic theories and semiring categories},
%     journal = {ArXiv e-prints},
%     eprint = {http://arxiv.org/abs/1606.05606},
%     year = {2016},
% }

\bib{blanc}{article}{
    author={Blanc, Anthony},
    title={Topological K-theory of complex noncommutative spaces},
    journal={Compos. Math.},
    volume={152},
    date={2016},
    number={3},
    pages={489--555},
    issn={0010-437X},
%     review={\MR{3477639}},
%     doi={10.1112/S0010437X15007617},
}

% % 
% % % \bib{bloch-ogus}{article}{
% % %     author={Bloch, Spencer},
% % %     author={Ogus, Arthur},
% % %     title={Gersten's conjecture and the homology of schemes},
% % %     journal={Ann. Sci. \'Ecole Norm. Sup. (4)},
% % %     volume={7},
% % %     date={1974},
% % %     pages={181--201 (1975)},
% % %     issn={0012-9593},
% % % %     review={\MR{0412191 (54 \#318)}},
% % % }

\bib{bgt1}{article}{
    author={Blumberg, Andrew J.},
    author={Gepner, David},
    author={Tabuada, Gon{\c{c}}alo},
    title={A universal characterization of higher algebraic
    $K$-theory},
    journal={Geom. Topol.},
    volume={17},
    date={2013},
    number={2},
    pages={733--838},
    issn={1465-3060},
%     review={\MR{3070515}},
%     doi={10.2140/gt.2013.17.733},
}

% \bib{bgt3}{article}{
%     author={Blumberg, Andrew J.},
%     author={Gepner, David},
%     author={Tabuada, Gon{\c{c}}alo},
%     title={$K$-theory of endomorphisms via noncommutative motives},
%     journal={Trans. Amer. Math. Soc.},
%     volume={368},
%     date={2016},
%     number={2},
%     pages={1435--1465},
%     issn={0002-9947},
% %     review={\MR{3430369}},
% %     doi={10.1090/tran/6507},
% }

% \bib{bgt4}{article}{
%     author={Blumberg, Andrew J.},
%     author={Gepner, David},
%     author={Tabuada, Gon{\c{c}}alo},
%     title={Homotopy $K$-theory of noncommutative motives},
%     note={Forthcoming},
% }

% \bib{blumberg-mandell}{article}{
%     author={Blumberg, Andrew J.},
%     author={Mandell, Michael A.},
%     title={The localization sequence for the algebraic $K$-theory of topological $K$-theory},
%     journal={Acta Math.},
%     volume={200},
%     date={2008},
%     number={2},
%     pages={155--179},
%     issn={0001-5962},
% %     review={\MR{2413133 (2009f:19003)}},
% %     doi={10.1007/s11511-008-0025-4},
% }
% 
% \bib{blumberg-mandell-thh}{article}{
%     author={Blumberg, Andrew J.},
%     author={Mandell, Michael A.},
%     title={Localization theorems in topological Hochschild homology
%     and
%     topological cyclic homology},
%     journal={Geom. Topol.},
%     volume={16},
%     date={2012},
%     number={2},
%     pages={1053--1120},
%     issn={1465-3060},
% %     review={\MR{2928988}},
% %     doi={10.2140/gt.2012.16.1053},
% }

% % \bib{bott}{article}{
% %     author={Bott, Raoul},
% %     title={The space of loops on a Lie group},
% %     journal={Michigan Math. J.},
% %     volume={5},
% %     date={1958},
% %     pages={35--61},
% %     issn={0026-2285},
% % }
% % 

% \bib{BousfieldFriedlander}{article}{
%     AUTHOR = {Bousfield, A. K.},
%     author = {Friedlander, E. M.},
%      TITLE = {Homotopy theory of {$\Gamma $}-spaces, spectra, and
%               bisimplicial sets},
%  BOOKTITLE = {Geometric applications of homotopy theory (Proc. Conf.,
%               Evanston, Ill., 1977), II},
%     SERIES = {Lecture Notes in Math.},
%     VOLUME = {658},
%      PAGES = {80--130},
%  PUBLISHER = {Springer},
%    ADDRESS = {Berlin},
%       YEAR = {1978},
% %    MRCLASS = {55P65 (55P42)},
% %   MRNUMBER = {MR513569 (80e:55021)},
% % MRREVIEWER = {D. W. Anderson},
% }

% % \bib{bousfield-kan}{book}{
% %     author={Bousfield, A. K.},
% %     author={Kan, D. M.},
% %     title={Homotopy limits, completions and localizations},
% %     series={Lecture Notes in Mathematics, Vol. 304},
% %     publisher={Springer-Verlag},
% %     place={Berlin},
% %     date={1972},
% %     pages={v+348},
% % %     review={\MR{0365573 (51 \#1825)}},
% % }

% \bib{braunling-groechenig-wolfson}{article}{
%     author={Braunling, Oliver},
%     author={Groechenig, Michael},
%     author={Wolfson, Jesse},
%     title={Tate objects in exact categories (with an appendix by Jan
%         \v{S}\v{t}ov\'i\v{c}ek and Jan
%     Trlifaj)},
%     journal = {ArXiv e-prints},
%     eprint = {http://arxiv.org/abs/1402.4969},
%     year = {2014},
% }


% % 
% % % \bib{brown}{article}{
% % %     author={Brown, Edgar H., Jr.},
% % %     title={The cohomology of $B{\rm SO}_{n}$ and $B{\rm O}_{n}$ with
% % %     integer coefficients},
% % %     journal={Proc. Amer. Math. Soc.},
% % %     volume={85},
% % %     date={1982},
% % %     number={2},
% % %     pages={283--288},
% % %     issn={0002-9939},
% % % %     review={\MR{652459 (83d:55015)}},
% % % %     doi={10.2307/2044298},
% % % }

% \bib{buchweitz}{article}{
%     author={Buchweitz, Ragnar-Olaf},
%     title={Maximal Cohen-Macaulay modules and Tate-cohomology over Gorenstein rings},
%     date={1986},
%     eprint={http://hdl.handle.net/1807/16682},
% }

% \bib{carlsson-products}{article}{
%     author={Carlsson, Gunnar},
%     title={On the algebraic $K$-theory of infinite product categories},
%     journal={$K$-Theory},
%     volume={9},
%     date={1995},
%     number={4},
%     pages={305--322},
%     issn={0920-3036},
% %     review={\MR{1351941}},
% %     doi={10.1007/BF00961467},
% }

% % 
% % \bib{cartan}{article}{
% %     author = {Cartan, H.},
% %     title = {D{\'e}termination des alg{\`e}bres {$\Hoh_*(\pi,n;\ZZ)$}},
% %     journal = {S{\'e}minaire H. Cartan},
% %     volume = {7},
% %     number = {1},
% %     pages = {11-01--11-24},
% %     publisher = {Secr{\'e}tariat math{\'e}matique},
% %     address = {Paris},
% %     year = {1954/1955},
% % }
% % 
% % \bib{chernousov-panin}{article}{
% %     author={Chernousov, Vladimir},
% %     author={Panin, Ivan},
% %     title={Purity of $G_2$-torsors},
% %     journal={C. R. Math. Acad. Sci. Paris},
% %     volume={345},
% %     date={2007},
% %     number={6},
% %     pages={307--312},
% %     issn={1631-073X},
% % %     review={\MR{2359087 (2008j:20146)}},
% % %     doi={10.1016/j.crma.2007.07.018},
% % }

% \bib{cisinski-tabuada}{article}{
%     author={Cisinski, Denis-Charles},
%     author={Tabuada, Gon\c{c}alo},
%     title={Non-connective $K$-theory via universal invariants},
%     journal={Compos. Math.},
%     volume={147},
%     date={2011},
%     number={4},
%     pages={1281--1320},
%     issn={0010-437X},
% %     review={\MR{2822869}},
% %     doi={10.1112/S0010437X11005380},
% }

\bib{cohn}{article}{
    author = {Cohn, Lee},
    title = {Differential Graded Categories are k-linear Stable Infinity Categories},
    journal = {ArXiv e-prints},
    eprint = {http://arxiv.org/abs/1308.2587},
    year = {2013},
}

% % 
% % \bib{colliot-thelene-sansuc}{article}{
% %     author={Colliot-Th{\'e}l{\`e}ne, J.-L.},
% %     author={Sansuc, J.-J.},
% %     title={Fibr\'es quadratiques et composantes connexes r\'eelles},
% %     journal={Math. Ann.},
% %     volume={244},
% %     date={1979},
% %     number={2},
% %     pages={105--134},
% %     issn={0025-5831},
% % %     review={\MR{550842 (81c:14010)}},
% % %     doi={10.1007/BF01420486},
% % }
% % 
% % \bib{colliot-thelene-birational}{article}{
% %     author={Colliot-Th{\'e}l{\`e}ne, J.-L.},
% %     title={Birational invariants, purity and the Gersten conjecture},
% %     conference={
% %         address={Santa Barbara, CA},
% %         date={1992},
% %     },
% %     book={
% %         series={Proc. Sympos. Pure Math.},
% %         volume={58},
% %         publisher={Amer. Math. Soc.},
% %         place={Providence, RI},
% %     },
% %     date={1995},
% %     pages={1--64},
% % %     review={\MR{1327280 (96c:14016)}},
% % }
% % 
% % % \bib{colliot}{article}{
% % %     author={Colliot-Th{\'e}l{\`e}ne, Jean-Louis},
% % %     title={Exposant et indice d'alg\`ebres simples centrales non ramifi\'ees},
% % %     note={With an appendix by Ofer Gabber},
% % %     journal={Enseign. Math. (2)},
% % %     volume={48},
% % %     date={2002},
% % %     number={1-2},
% % %     pages={127--146},
% % %     issn={0013-8584},
% % % }
% % 
% % \bib{colliot-thelene-voisin}{article}{
% %     author={Colliot-Th{\'e}l{\`e}ne, Jean-Louis},
% %     author={Voisin, Claire},
% %     title={Cohomologie non ramifi\'ee et conjecture de Hodge enti\`ere},
% %     journal={Duke Math. J.},
% %     volume={161},
% %     date={2012},
% %     number={5},
% %     pages={735--801},
% %     issn={0012-7094},
% % %     review={\MR{2904092}},
% % %     doi={10.1215/00127094-1548389},
% % }
% \bib{cortinas-haesemeyer-weibel}{article}{
%     author={Corti{\~n}as, G.},
%     author={Haesemeyer, C.},
%     author={Weibel, C.},
%     title={$K$-regularity, $cdh$-fibrant Hochschild homology, and a
%     conjecture of Vorst},
%     journal={J. Amer. Math. Soc.},
%     volume={21},
%     date={2008},
%     number={2},
%     pages={547--561},
%     issn={0894-0347},
% %     review={\MR{2373359}},
% %     doi={10.1090/S0894-0347-07-00571-1},
% }

% \bib{curtis-reiner}{book}{
%     author={Curtis, Charles W.},
%     author={Reiner, Irving},
%     title={Representation theory of finite groups and associative
%     algebras},
%     series={Pure and Applied Mathematics, Vol. XI},
%     publisher={Interscience Publishers, a division of John Wiley
%     \& Sons, New
%     York-London},
%     date={1962},
%     pages={xiv+685},
% %     review={\MR{0144979}},
% }

% % 
% % % \bib{dejong}{article}{
% % %     author={de Jong, Johan},
% % %     title={The period-index problem for the Brauer group of an algebraic surface},
% % %     journal={Duke Math. J.},
% % %     volume={123},
% % %     date={2004},
% % %     number={1},
% % %     pages={71--94},
% % %     issn={0012-7094},
% % % %     review={\MR{2060023 (2005e:14025)}},
% % % %     doi={10.1215/S0012-7094-04-12313-9},
% % % }
% % 
% % \bib{dejong-gabber}{article}{
% %     author={de Jong, A. J.},
% %     title={A result of Gabber},
% %     eprint={http://www.math.columbia.edu/~dejong/},
% % }
% % 
% % 
% % % \bib{dejong-starr}{article}{
% % %     author={Starr, Jason},
% % %     author={de Jong, Johan},
% % %     title={Almost proper GIT-stacks and discriminant avoidance},
% % %     journal={Doc. Math.},
% % %     volume={15},
% % %     date={2010},
% % %     pages={957--972},
% % %     issn={1431-0635},
% % % %     review={\MR{2745688 (2012e:14073)}},
% % % }
% % 
% % \bib{demeyer}{article}{
% %     author={DeMeyer, F. R.},
% %     title={Projective modules over central separable algebras},
% %     journal={Canad. J. Math.},
% %     volume={21},
% %     date={1969},
% %     pages={39--43},
% %     issn={0008-414X},
% % %     review={\MR{0234987 (38 \#3299)}},
% % }

\bib{dugger-isaksen}{article}{
    author={Dugger, Daniel},
    author={Isaksen, Daniel C.},
    title={Topological hypercovers and $\Bbb A^1$-realizations},
    journal={Math. Z.},
    volume={246},
    date={2004},
    number={4},
    pages={667--689},
    issn={0025-5874},
%     review={\MR{2045835}},
%     doi={10.1007/s00209-003-0607-y},
}

% %
% \bib{DundasGoodwillieMcCarthy}{book}{
%     AUTHOR = {Dundas, Bj{\o}rn I.},
%     author ={ Goodwillie, T.},
%     author ={McCarthy, R.},
%      TITLE = {The local structure of algebraic {K}-theory},
%     SERIES = {Algebra and Applications},
%     VOLUME = {18},
%  PUBLISHER = {Springer-Verlag London, Ltd., London},
%       YEAR = {2013},
%      PAGES = {xvi+435},
%       ISBN = {978-1-4471-4392-5; 978-1-4471-4393-2},
% %    MRCLASS = {19-02 (16E40 19D55 55-02 55N99)},
% %   MRNUMBER = {3013261},
% % MRREVIEWER = {Charles Weibel},
% }

% \bib{dgi}{article}{
%     author={Dwyer, W. G.},
%     author={Greenlees, J. P. C.},
%     author={Iyengar, S.},
%     title={Duality in algebra and topology},
%     journal={Adv. Math.},
%     volume={200},
%     date={2006},
%     number={2},
%     pages={357--402},
%     issn={0001-8708},
% %     review={\MR{2200850}},
% %     doi={10.1016/j.aim.2005.11.004},
% }

% % \bib{ekedahl}{article}{
% %     author = {Ekedahl, Torsten},
% %     title = {Approximating classifying spaces by smooth projective varieties},
% %     journal = {ArXiv e-prints},
% %     eprint =  {http://arxiv.org/abs/0905.1538},
% %     year = {2009},
% % }
% % 
% \bib{EKMM}{book} {
%     AUTHOR = {Elmendorf, A. D},
%     author ={Kriz, I.},
%     author ={Mandell, M. A.},
%     author = {May, J.P.},
%      TITLE = {Rings, modules, and algebras in stable homotopy theory},
%     SERIES = {Mathematical Surveys and Monographs},
%     VOLUME = {47},
%       NOTE = {With an appendix by M. Cole},
%  PUBLISHER = {American Mathematical Society, Providence, RI},
%       YEAR = {1997},
%      PAGES = {xii+249},
%       ISBN = {0-8218-0638-6},
% %    MRCLASS = {55N20 (19D10 19D55 55P42 55T25)},
% %   MRNUMBER = {1417719 (97h:55006)},
% % MRREVIEWER = {Donald M. Davis},
% }

% \bib{farrell-jones}{article}{
%     author={Farrell, F. T.},
%     author={Jones, L. E.},
%     title={The lower algebraic $K$-theory of virtually infinite cyclic
%     groups},
%     journal={$K$-Theory},
%     volume={9},
%     date={1995},
%     number={1},
%     pages={13--30},
%     issn={0920-3036},
% %     review={\MR{1340838}},
% %     doi={10.1007/BF00965457},
% }
% 
% 
% % % \bib{fedorov-panin}{article}{
% % %     author = {Fedorov, Roman},
% % %     author = {Panin, Ivan},
% % %     title = {Proof of Grothendieck-Serre conjecture on principal bundles over regular local rings containing infinite fields},
% % %     journal = {ArXiv e-prints},
% % %     eprint =  {http://arxiv.org/abs/1211.2678},
% % %     year = {2012},
% % % }
% 
% \bib{freyd}{article}{
%     author={Freyd, Peter},
%     title={Representations in abelian categories},
%     conference={
%     title={Proc. Conf. Categorical Algebra},
%     address={La Jolla, Calif.},
%     date={1965},
%     },
%     book={
%     publisher={Springer, New
%     York},
%     },
%     date={1966},
%     pages={95--120},
% %     review={\MR{0209333}},
% }

\bib{friedlander-walker-comparing}{article}{
    author={Friedlander, Eric M.},
    author={Walker, Mark E.},
    title={Comparing $K$-theories for complex varieties},
    journal={Amer. J. Math.},
    volume={123},
    date={2001},
    number={5},
    pages={779--810},
    issn={0002-9327},
%     review={\MR{1854111}},
}

\bib{friedlander-walker}{article}{
    author={Friedlander, Eric M.},
    author={Walker, Mark E.},
    title={Semi-topological $K$-theory using function complexes},
    journal={Topology},
    volume={41},
    date={2002},
    number={3},
    pages={591--644},
    issn={0040-9383},
%     review={\MR{1910042}},
%     doi={10.1016/S0040-9383(01)00023-4},
}

\bib{friedlander-walker-handbook}{article}{
    author={Friedlander, Eric M.},
    author={Walker, Mark E.},
    title={Semi-topological $K$-theory},
    conference={
    title={Handbook of $K$-theory. Vol. 1, 2},
    },
    book={
    publisher={Springer, Berlin},
    },
    date={2005},
    pages={877--924},
%     review={\MR{2181835}},
%     doi={10.1007/978-3-540-27855-9_17},
}

% % 
% % \bib{fujiwara}{article}{
% %     author={Fujiwara, Kazuhiro},
% %     title={A proof of the absolute purity conjecture (after Gabber)},
% %     conference={
% %         title={Algebraic geometry 2000, Azumino (Hotaka)},
% %     },
% %     book={
% %         series={Adv. Stud. Pure Math.},
% %         volume={36},
% %         publisher={Math. Soc. Japan},
% %         place={Tokyo},
% %     },
% %     date={2002},
% %     pages={153--183},
% % %     review={\MR{1971516 (2004d:14015)}},
% % }
% % 
% % \bib{gabber}{article}{
% %     author={Gabber, Ofer},
% %     title={Some theorems on Azumaya algebras},
% %     conference={
% %     title={The Brauer group},
% %     address={Sem., Les Plans-sur-Bex},
% %     date={1980},
% %     },
% %     book={
% %     series={Lecture Notes in Math.},
% %     volume={844},
% %     publisher={Springer},
% %     place={Berlin},
% %     },
% %     date={1981},
% %     pages={129--209},
% % %     review={\MR{611868 (83d:13004)}},
% % }
% % 
% % \bib{gabber-note}{article}{
% %     author={Gabber, Ofer},
% %     title={A note on the unramified Brauer group and purity},
% %     journal={Manuscripta Math.},
% %     volume={95},
% %     date={1998},
% %     number={1},
% %     pages={107--115},
% %     issn={0025-2611},
% % %     review={\MR{1492372 (98m:14018)}},
% % %     doi={10.1007/BF02678018},
% % }

% \bib{gepner-groth-nikolaus}{article}{
%     author={Gepner, David},
%     author={Groth, Moritz},
%     author={Nikolaus, Thomas},
%     title={Universality of multiplicative infinite loop space
%     machines},
%     journal={Algebr. Geom. Topol.},
%     volume={15},
%     date={2015},
%     number={6},
%     pages={3107--3153},
%     issn={1472-2747},
% %     review={\MR{3450758}},
% %     doi={10.2140/agt.2015.15.3107},
% }

% \bib{gepner-haugseng}{article}{
%     author={Gepner, David},
%     author={Haugseng, Rune},
%     title={Enriched $\infty$-categories via non-symmetric
%     $\infty$-operads},
%     journal={Adv. Math.},
%     volume={279},
%     date={2015},
%     pages={575--716},
%     issn={0001-8708},
%     review={\MR{3345192}},
%     doi={10.1016/j.aim.2015.02.007},
% }
% % 
% % \bib{gille-szamuely}{book}{
% %     author={Gille, Philippe},
% %     author={Szamuely, Tam{\'a}s},
% %     title={Central simple algebras and Galois cohomology},
% %     series={Cambridge Studies in Advanced Mathematics},
% %     volume={101},
% %     publisher={Cambridge University Press},
% %     place={Cambridge},
% %     date={2006},
% %     pages={xii+343},
% %     isbn={978-0-521-86103-8},
% %     isbn={0-521-86103-9},
% %     % review={\MR{2266528 (2007k:16033)}},
% %     % doi={10.1017/CBO9780511607219},
% % }
% \bib{glasman}{article}{
%     author={Glasman, Saul},
%     title={Day convolution for infinity-categories},
%     journal = {ArXiv e-prints},
%     eprint = {http://arxiv.org/abs/1308.4940},
%     year = {2013},
% }

% \bib{glaz}{book}{
%     author={Glaz, Sarah},
%     title={Commutative coherent rings},
%     series={Lecture Notes in Mathematics},
%     volume={1371},
%     publisher={Springer-Verlag, Berlin},
%     date={1989},
%     pages={xii+347},
%     isbn={3-540-51115-6},
% %     review={\MR{999133}},
% %     doi={10.1007/BFb0084570},
% }

% % 
% % \bib{goresky-macpherson}{book}{
% %     author={Goresky, Mark},
% %     author={MacPherson, Robert},
% %     title={Stratified Morse theory},
% %     series={Ergebnisse der Mathematik und ihrer Grenzgebiete (3)},
% %     volume={14},
% %     publisher={Springer-Verlag},
% %     place={Berlin},
% %     date={1988},
% %     pages={xiv+272},
% %     isbn={3-540-17300-5},
% % %     review={\MR{932724 (90d:57039)}},
% % }
% % 
% % \bib{grothendieck-brauer-1}{article}{
% %     author={Grothendieck, Alexander},
% %     title={Le groupe de Brauer. I. Alg\`ebres d'Azumaya et interpr\'etations diverses},
% %     conference={
% %     title={S\'eminaire Bourbaki, Vol.\ 9},
% %     },
% %     book={
% %     publisher={Soc. Math. France},
% %     place={Paris},
% %     },
% %     date={1995},
% %     pages={Exp.\ No.\ 290, 199--219},
% % }
% 
% % \bib{grothendieck-brauer-2}{article}{
% %     author={Grothendieck, Alexander},
% %     title={Le groupe de Brauer. II. Th\'eorie cohomologique},
% %     conference={
% %         title={Dix Expos\'es sur la Cohomologie des Sch\'emas},
% %     },
% %     book={
% %         publisher={North-Holland},
% %         place={Amsterdam},
% %     },
% %     date={1968},
% %     pages={67--87},
% % %     review={\MR{0244270 (39 \#5586b)}},
% % }
% % % 

% \bib{happel}{article}{
%     author={Happel, Dieter},
%     title={On the derived category of a finite-dimensional algebra},
%     journal={Comment. Math. Helv.},
%     volume={62},
%     date={1987},
%     number={3},
%     pages={339--389},
%     issn={0010-2571},
% %     review={\MR{910167}},
% %     doi={10.1007/BF02564452},
% }

% \bib{hatcher}{book}{
%     author={Hatcher, Allen},
%     title={Algebraic topology},
%     publisher={Cambridge University Press},
%     place={Cambridge},
%     date={2002},
%     pages={xii+544},
% %     isbn={0-521-79160-X},
% %     isbn={0-521-79540-0},
% %     review={\MR{1867354 (2002k:55001)}},
% }

% \bib{heller}{article}{
%     author={Heller, Alex},
%     title={The loop-space functor in homological algebra},
%     journal={Trans. Amer. Math. Soc.},
%     volume={96},
%     date={1960},
%     pages={382--394},
%     issn={0002-9947},
% %     review={\MR{0116045}},
% }

% \bib{heller-st}{article}{
%     author={Heller, Jeremiah},
%     title={Motivic strict ring spectra representing semi-topological
%     cohomology theories},
%     journal={Homology Homotopy Appl.},
%     volume={17},
%     date={2015},
%     number={2},
%     pages={107--135},
%     issn={1532-0073},
% %     review={\MR{3421465}},
% %     doi={10.4310/HHA.2015.v17.n2.a7},
% }
% 
% \bib{hennion}{article}{
%     author={Hennion, Benjamin},
%     title={Tate objects in stable $(\infty,1)$-categories},
%     journal = {ArXiv e-prints},
%     eprint = {http://arxiv.org/abs/1606.05527},
%     year = {2016},
% }
% 
% \bib{hennion-porta-vezzosi}{article}{
%     author={Hennion, Benjamin},
%     author={Porta, Mauro},
%     author={Vezzosi, Gabriele},
%     title={Formal glueing for non-linear flags},
%     journal = {ArXiv e-prints},
%     eprint = {http://arxiv.org/abs/1607.04503},
%     year = {2016},
% }

% \bib{hochschild-kostant-rosenberg}{article}{
%     author={Hochschild, G.},
%     author={Kostant, Bertram},
%     author={Rosenberg, Alex},
%     title={Differential forms on regular affine algebras},
%     journal={Trans. Amer. Math. Soc.},
%     volume={102},
%     date={1962},
%     pages={383--408},
%     issn={0002-9947},
% %     review={\MR{0142598 (26 \#167)}},
% }
% %  
% % \bib{hoobler}{article}{
% %     author={Hoobler, Raymond T.},
% %     title={A cohomological interpretation of Brauer groups of rings},
% %     journal={Pacific J. Math.},
% %     volume={86},
% %     date={1980},
% %     number={1},
% %     pages={89--92},
% %     issn={0030-8730},
% % %     review={\MR{586870 (81j:13003)}},
% % }
% % 

\bib{hornbostel-yagunov}{article}{
    author={Hornbostel, Jens},
    author={Yagunov, Serge},
    title={Rigidity for Henselian local rings and $\mathbb{A}^1$-representable theories},
    journal={Math. Z.},
    volume={255},
    date={2007},
    number={2},
    pages={437--449},
    issn={0025-5874},
%     review={\MR{2262740}},
%     doi={10.1007/s00209-006-0049-4},
}

% \bib{hovey}{book}{
%     author={Hovey, Mark},
%     title={Model categories},
%     series={Mathematical Surveys and Monographs},
%     volume={63},
%     publisher={American Mathematical Society, Providence, RI},
%     date={1999},
%     pages={xii+209},
%     isbn={0-8218-1359-5},
% %     review={\MR{1650134}},
% }

% \bib{hoyois-cdh}{article}{
% 	author = {{Hoyois}, M.},
% 	title = {Cdh descent in equivariant homotopy K-theory},
% 	journal = {ArXiv e-prints},
% 	eprint = {http://arxiv.org/abs/1604.06410},
% 	year = {2016},
% %	adsurl = {http://adsabs.harvard.edu/abs/2016arXiv160406410H},
% %	adsnote = {Provided by the SAO/NASA Astrophysics Data System}
% }

% \bib{husemoller}{book}{
%     author={Husemoller, Dale},
%     title={Fibre bundles},
%     edition={2},
%     note={Graduate Texts in Mathematics, No. 20},
%     publisher={Springer-Verlag},
%     place={New York},
%     date={1975},
%     pages={xv+327},
% }

% \bib{huybrechts}{book}{
%     author={Huybrechts, D.},
%     title={Fourier-Mukai transforms in algebraic geometry},
%     series={Oxford Mathematical Monographs},
%     publisher={The Clarendon Press, Oxford University Press,
%     Oxford},
%     date={2006},
%     pages={viii+307},
%     isbn={978-0-19-929686-6},
%     isbn={0-19-929686-3},
% %     review={\MR{2244106}},
% %     doi={10.1093/acprof:oso/9780199296866.001.0001},
% }

% % 
% % \bib{jackowski-mcclure-oliver-1}{article}{
% %     author={Jackowski, Stefan},
% %     author={McClure, James},
% %     author={Oliver, Bob},
% %     title={Homotopy classification of self-maps of $BG$ via $G$-actions. I},
% %     journal={Ann. of Math. (2)},
% %     volume={135},
% %     date={1992},
% %     number={1},
% %     pages={183--226},
% %     issn={0003-486X},
% % %     review={\MR{1147962 (93e:55019a)}},
% % %     doi={10.2307/2946568},
% % }
% % 
% % % \bib{jardine}{article}{
% % %    author={Jardine, J. F.},
% % %    title={Simplicial presheaves},
% % %    journal={J. Pure Appl. Algebra},
% % %    volume={47},
% % %    date={1987},
% % %    number={1},
% % %    pages={35--87},
% % %    issn={0022-4049},
% % %   % review={\MR{906403 (88j:18005)}},
% % %   % doi={10.1016/0022-4049(87)90100-9},
% % % }

% \bib{jorgensen-sega}{article}{
%     author={Jorgensen, David A.},
%     author={{\c{S}}ega, Liana M.},
%     title={Independence of the total reflexivity conditions for
%     modules},
%     journal={Algebr. Represent. Theory},
%     volume={9},
%     date={2006},
%     number={2},
%     pages={217--226},
%     issn={1386-923X},
% %     review={\MR{2238367}},
% %     doi={10.1007/s10468-005-0559-5},
% }

% % 
\bib{jouanolou}{article}{
    author={Jouanolou, J. P.},
    title={Une suite exacte de Mayer-Vietoris en $K$-th\'eorie alg\'ebrique},
    conference={
    title={Algebraic $K$-theory, I: Higher $K$-theories (Proc. Conf.,
    Battelle Memorial Inst., Seattle, Wash., 1972)},
    },
    book={
    publisher={Springer},
    place={Berlin},
    },
    date={1973},
    pages={293--316. Lecture
    Notes in Math., Vol.
    341},
%     review={\MR{0409476 (53 \#13231)}},
}
% % 
% % % \bib{kameko-yagita}{article}{
% % %     author={Kameko, Masaki},
% % %     author={Yagita, Nobuaki},
% % %     title={The Brown-Peterson cohomology of the classifying spaces of the
% % %     projective unitary groups ${\rm PU}(p)$ and exceptional Lie groups},
% % %     journal={Trans. Amer. Math. Soc.},
% % %     volume={360},
% % %     date={2008},
% % %     number={5},
% % %     pages={2265--2284},
% % %     issn={0002-9947},
% % % }

% \bib{keller-nicolas}{article}{
%     author={Keller, Bernhard},
%     author={Nicol{\'a}s, Pedro},
%     title={Weight structures and simple dg modules for positive dg
%     algebras},
%     journal={Int. Math. Res. Not. IMRN},
%     date={2013},
%     number={5},
%     pages={1028--1078},
%     issn={1073-7928},
% %     review={\MR{3031826}},
% }

% % 
% % % \bib{kervaire}{article}{
% % %     author={Kervaire, Michel A.},
% % %     title={Some nonstable homotopy groups of Lie groups},
% % %     journal={Illinois J. Math.},
% % %     volume={4},
% % %     date={1960},
% % %     pages={161--169},
% % %     issn={0019-2082},
% % % %     review={\MR{0113237 (22 \#4075)}},
% % % }
% \bib{khan-thesis}{article}{
%     author={Khan, Adeel},
%     title={The motivic homotopy theory of derived schemes},
%     note={Forthcoming},
% }

% \bib{krause}{article}{
%     author={Krause, Henning},
%     title={Deriving Auslander's formula},
%     journal={Doc. Math.},
%     volume={20},
%     date={2015},
%     pages={669--688},
%     issn={1431-0635},
% %     review={\MR{3398723}},
% }
% % 
% % % \bib{lieblich}{article}{
% % %     author={Lieblich, Max},
% % %     title={Twisted sheaves and the period-index problem},
% % %     journal={Compos. Math.},
% % %     volume={144},
% % %     date={2008},
% % %     number={1},
% % %     pages={1--31},
% % %     issn={0010-437X},
% % % %     review={\MR{2388554 (2009b:14033)}},
% % % %     doi={10.1112/S0010437X07003144},
% % % }
% % 

% \bib{luck-reich}{article}{
%     author={L{\"u}ck, Wolfgang},
%     author={Reich, Holger},
%     title={The Baum-Connes and the Farrell-Jones conjectures in $K$-
%     and
%     $L$-theory},
%     conference={
%     title={Handbook of $K$-theory. Vol. 1, 2},
%     },
%     book={
%     publisher={Springer, Berlin},
%     },
%     date={2005},
%     pages={703--842},
% %     review={\MR{2181833}},
% %     doi={10.1007/978-3-540-27855-9_15},
% }

\bib{htt}{book}{
      author={Lurie, Jacob},
       title={Higher topos theory},
      series={Annals of Mathematics Studies},
   publisher={Princeton University Press},
     address={Princeton, NJ},
        date={2009},
      volume={170},
        ISBN={978-0-691-14049-0; 0-691-14049-9},
%       review={\MR{MR2522659}},
}

% \bib{ha}{article}{
%     author={Lurie, Jacob},
%     title={Higher algebra},
%     date={2012},
%     eprint={http://www.math.harvard.edu/~lurie/},
%     note={Version dated 10 March 2016},
% }

% \bib{sag}{article}{
%     author={Lurie, Jacob},
%     title={Spectral algebraic geometry},
%     date={2016},
%     eprint={http://www.math.harvard.edu/~lurie/},
%     note={Version dated 13 October 2016},
% }
% 
% \bib{dag1}{unpublished}{
%     author={Lurie, Jacob},
%     title={Derived algebraic geometry I: stable $\infty$-categories},
%     date={2009},
%     eprint={http://www.math.harvard.edu/~lurie/papers/DAG-I.pdf},
%     note={Version dated 8 October 2009},
% }

% \bib{dag11}{unpublished}{
%       author={Lurie, Jacob},
%       title={Derived algebraic geometry XI: descent theorems},
%         date={2011},
%   note={\texttt{\href{http://www.math.harvard.edu/~lurie/}{http://www.math.harvard.edu/\textasciitilde lurie/}}},
% }

% \bib{mathew-galois}{article}{
%     author={Mathew, Akhil},
%     title={The Galois group of a stable homotopy theory},
%     journal={Adv. Math.},
%     volume={291},
%     date={2016},
%     pages={403--541},
%     issn={0001-8708},
% %     review={\MR{3459022}},
% %     doi={10.1016/j.aim.2015.12.017},
% }

% \bib{mazza-voevodsky-weibel}{book}{
% author = {Mazza, Carlo},
%  AUTHOR = {Voevodsky, Vladimir},
%  AUTHOR = {Weibel, Charles},
%      TITLE = {Lecture notes on motivic cohomology},
%     SERIES = {Clay Mathematics Monographs},
%     VOLUME = {2},
%  PUBLISHER = {American Mathematical Society, Providence, RI; Clay
%               Mathematics Institute, Cambridge, MA},
%       YEAR = {2006},
%      PAGES = {xiv+216},
%       %ISBN = {978-0-8218-3847-1; 0-8218-3847-4},
%    %MRCLASS = {14F42 (19E15)},
%   %MRNUMBER = {2242284},
% %MRREVIEWER = {Thomas Geisser},
% }

% % % \bib{mccleary}{book}{
% % %     author={McCleary, John},
% % %     title={A user's guide to spectral sequences},
% % %     series={Cambridge Studies in Advanced Mathematics},
% % %     volume={58},
% % %     edition={2},
% % %     publisher={Cambridge University Press},
% % %     place={Cambridge},
% % %     date={2001},
% % %     pages={xvi+561},
% % %     isbn={0-521-56759-9},
% % % }
% % 

% \bib{mcconnell-robson}{book}{
%     author={McConnell, J. C.},
%     author={Robson, J. C.},
%     title={Noncommutative Noetherian rings},
%     series={Graduate Studies in Mathematics},
%     volume={30},
%     edition={Revised edition},
%     note={With the cooperation of L. W. Small},
%     publisher={American Mathematical Society,
%     Providence, RI},
%     date={2001},
%     pages={xx+636},
%     isbn={0-8218-2169-5},
% %     review={\MR{1811901}},
% %     doi={10.1090/gsm/030},
% }

% % % \bib{milnor-stasheff}{book}{
% % %     author={Milnor, John W.},
% % %     author={Stasheff, James D.},
% % %     title={Characteristic classes},
% % %     note={Annals of Mathematics Studies, No. 76},
% % %     publisher={Princeton University Press},
% % %     place={Princeton, N. J.},
% % %     date={1974},
% % %     pages={vii+331},
% % % }
% 
  \bib{morel-voevodsky}{article}{
      author={Morel, Fabien},
      author={Voevodsky, Vladimir},
      title={${\bf A}^1$-homotopy theory of schemes},
      journal={Inst. Hautes \'Etudes Sci. Publ. Math.},
      number={90},
      date={1999},
      pages={45--143 (2001)},
      issn={0073-8301},
 %      review={\MR{1813224 (2002f:14029)}},
  }
% 
% \bib{murthy-survey}{article}{
%     author={Murthy, M. Pavaman},
%     title={A survey of obstruction theory for projective modules of top rank},
%     conference={
%         title={Algebra, $K$-theory, groups, and education},
%         address={New York},
%         date={1997},
%     },
%     book={
%         series={Contemp. Math.},
%         volume={243},
%         publisher={Amer.  Math. Soc., Providence, RI},
%     },
%     date={1999},
%     pages={153--174},
% %     review={\MR{1732046 (2000m:13013)}},
% %     doi={10.1090/conm/243/03692},
% }

% \bib{neeman-triangulated}{article}{
%     author={Neeman, Amnon},
%     title={The $K$-theory of triangulated categories},
%     conference={
%     title={Handbook of $K$-theory. Vol. 1, 2},
%     },
%     book={
%     publisher={Springer, Berlin},
%     },
%     date={2005},
%     pages={1011--1078},
% %     review={\MR{2181838}},
% %     doi={10.1007/978-3-540-27855-9_20},
% }

% 
% % 
% % \bib{ojanguren}{article}{
% %     author = {Ojanguren, Manuel},
% %     title = {Wedderburn's theorem for regular local rings},
% %     journal = {Linear Algebraic Groups Preprint Server},
% %     eprint = {http://www.mathematik.uni-bielefeld.de/LAG/man/500.html},
% %     year = {2013},
% % }
% % 
% % \bib{ojanguren-panin}{article}{
% %     author={Ojanguren, Manuel},
% %     author={Panin, Ivan},
% %     title={A purity theorem for the Witt group},
% %     journal={Ann. Sci. \'Ecole Norm. Sup. (4)},
% %     volume={32},
% %     date={1999},
% %     number={1},
% %     pages={71--86},
% %     issn={0012-9593},
% % %     review={\MR{1670591 (2000a:11057)}},
% % %     doi={10.1016/S0012-9593(99)80009-3},
% % }
% % 
% % \bib{panin-purity}{article}{
% %     author={Panin, I. A.},
% %     title={Purity conjecture for reductive groups},
% %     journal={Vestnik St. Petersburg Univ. Math.},
% %     volume={43},
% %     date={2010},
% %     number={1},
% %     pages={44--48},
% %     issn={1063-4541},
% % %     review={\MR{2662409 (2011j:14100)}},
% % %     doi={10.3103/S1063454110010085},
% % }
% % 
% % \bib{parimala-sridharan-2}{article}{
% %     author={Parimala, R.},
% %     author={Sridharan, R.},
% %     title={Nonsurjectivity of the Clifford invariant map},
% %     note={K. G. Ramanathan memorial issue},
% %     journal={Proc. Indian Acad. Sci. Math. Sci.},
% %     volume={104},
% %     date={1994},
% %     number={1},
% %     pages={49--56},
% %     issn={0253-4142},
% % %     review={\MR{1280057 (95g:11030)}},
% % %     doi={10.1007/BF02830873},
% % }
% \bib{popescu-1}{article}{
%     author={Popescu,
%     Dorin},
%     title={General
%     N\'eron
%     desingularization},
%     journal={Nagoya
%     Math.
%     J.},
%     volume={100},
%     date={1985},
%     pages={97--126},
%     issn={0027-7630},
% %     review={\MR{818160 (87f:13019)}},
% }
% 
% \bib{popescu-2}{article}{
%     author={Popescu, Dorin},
%     title={General N\'eron
%     desingularization and
%     approximation},
%     journal={Nagoya Math. J.},
%     volume={104},
%     date={1986},
%     pages={85--115},
%     issn={0027-7630},
% %     review={\MR{868439 (88a:14007)}},
% }
% 
% \bib{popescu-3}{article}{
%     author={Popescu, Dorin},
%     title={Letter to the editor: ``General N\'eron desingularization and
%     approximation'' [Nagoya Math.\ J.\ {\bf 104} (1986), 85--115;
%     MR0868439
%     (88a:14007)]},
%     journal={Nagoya Math. J.},
%     volume={118},
%     date={1990},
%     pages={45--53},
%     issn={0027-7630},
% %     review={\MR{1060701 (91e:14003)}},
% }


% % 
% % % \bib{puttmann-rigas}{article}{
% % %     author={P{\"u}ttmann, Thomas},
% % %     author={Rigas, A.},
% % %     title={Presentations of the first homotopy groups of the unitary groups},
% % %     journal={Comment. Math. Helv.},
% % %     volume={78},
% % %     date={2003},
% % %     number={3},
% % %     pages={648--662},
% % %     issn={0010-2571},
% % % }

% \bib{quillen}{article}{
%     author={Quillen, Daniel},
%     title={Higher algebraic $K$-theory. I},
%     conference={
%         title={Algebraic $K$-theory, I: Higher $K$-theories (Proc. Conf.,
%         Battelle Memorial Inst., Seattle, Wash., 1972)},
%     },
%     book={
%         publisher={Springer},
%         place={Berlin},
%     },
%     date={1973},
%     pages={85--147. Lecture Notes in Math., Vol.  341},
% %     review={\MR{0338129 (49 \#2895)}},
% }
% % 
% % % \bib{raynaud}{article}{
% % %    author={Raynaud, Mich{\`e}le},
% % %    title={Modules projectifs universels},
% % %    journal={Invent. Math.},
% % %    volume={6},
% % %    date={1968},
% % %    pages={1--26},
% % %    issn={0020-9910},
% % %  %  review={\MR{0236164 (38 \#4462)}},
% % % }
% % 

\bib{robalo}{article}{
    author={Robalo, Marco},
    title={$K$-theory and the bridge from motives to noncommutative
    motives},
    journal={Adv. Math.},
    volume={269},
    date={2015},
    pages={399--550},
    issn={0001-8708},
%     review={\MR{3281141}},
%     doi={10.1016/j.aim.2014.10.011},
}

% \bib{renault}{article}{
%     author={Renault, Guy},
%     title={Sur les anneaux de groupes},
%     journal={C. R. Acad. Sci. Paris S\'er. A-B},
%     volume={273},
%     date={1971},
%     pages={A84--A87},
% %     review={\MR{0288189}},
% }
% 
% \bib{saito}{article}{
%     author={Saito, Sho},
%     title={On Previdi's delooping conjecture for $K$-theory},
%     journal={Algebra Number Theory},
%     volume={9},
%     date={2015},
%     number={1},
%     pages={1--11},
%     issn={1937-0652},
% %     review={\MR{3317759}},
% %     doi={10.2140/ant.2015.9.1},
% }

% % \bib{saltman}{book}{
% %     author={Saltman, David J.},
% %     title={Lectures on division algebras},
% %     series={CBMS Regional Conference Series in Mathematics},
% %     volume={94},
% %     publisher={Published by American Mathematical Society, Providence, RI},
% %     date={1999},
% %     pages={viii+120},
% %     isbn={0-8218-0979-2},
% % %     review={\MR{1692654 (2000f:16023)}},
% % }
% \bib{schlichting}{article}{
%     author={Schlichting, Marco},
%     title={Negative $K$-theory of derived categories},
%     journal={Math. Z.},
%     volume={253},
%     date={2006},
%     number={1},
%     pages={97--134},
%     issn={0025-5874},
% %     review={\MR{2206639 (2006i:19003)}},
% %     doi={10.1007/s00209-005-0889-3},
% }
% % 
% % \bib{schoen}{article}{
% %     author={Schoen, Chad},
% %     title={Complex varieties for which the Chow group mod $n$ is not finite},
% %     journal={J. Algebraic Geom.},
% %     volume={11},
% %     date={2002},
% %     number={1},
% %     pages={41--100},
% %     issn={1056-3911},
% % %     review={\MR{1865914 (2002h:14004)}},
% % %     doi={10.1090/S1056-3911-01-00291-0},
% % }

% \bib{segal}{article}{
%     author={Segal, Graeme},
%     title={Categories and cohomology theories},
%     journal={Topology},
%     volume={13},
%     date={1974},
%     pages={293--312},
%     issn={0040-9383},
% %     review={\MR{0353298}},
% %     doi={10.1016/0040-9383(74)90022-6},
% }

% \bib{tabuada-duke}{article}{
%     author={Tabuada, Gon\c{c}alo},
%     title={Higher $K$-theory via universal invariants},
%     journal={Duke Math. J.},
%     volume={145},
%     date={2008},
%     number={1},
%     pages={121--206},
%     issn={0012-7094},
% %     review={\MR{2451292}},
% %     doi={10.1215/00127094-2008-049},
% }

\bib{tabuada-kh}{article}{
    author={Tabuada, Gon\c{c}alo},
    title={$\bold{A}^1$-homotopy theory of noncommutative motives},
    journal={J. Noncommut. Geom.},
    volume={9},
    date={2015},
    number={3},
    pages={851--875},
    issn={1661-6952},
%     review={\MR{3420534}},
%     doi={10.4171/JNCG/210},
}

\bib{tabuada-a1-modl}{article}{
    author={Tabuada, Gon\c{c}alo},
    title={$\Bbb{A}^1$-homotopy invariance of algebraic $K$-theory with
    coefficients and du Val singularities},
    journal={Ann. K-Theory},
    volume={2},
    date={2017},
    number={1},
    pages={1--25},
    issn={2379-1683},
%     review={\MR{3599514}},
%     doi={10.2140/akt.2017.2.1},
}

% % 
% % % \bib{thom}{article}{
% % %     author={Thom, Ren{\'e}},
% % %     title={Quelques propri\'et\'es globales des vari\'et\'es
% % %     diff\'erentiables},
% % %     journal={Comment. Math. Helv.},
% % %     volume={28},
% % %     date={1954},
% % %     pages={17--86},
% % % }

\bib{thomason}{article}{
    author={Thomason, R. W.},
    title={Algebraic $K$-theory and \'etale cohomology},
    journal={Ann. Sci. \'Ecole Norm. Sup. (4)},
    volume={18},
    date={1985},
    number={3},
    pages={437--552},
    issn={0012-9593},
    review={\MR{826102}},
}

\bib{thomason-bott}{article}{
    author={Thomason, R. W.},
    title={Bott stability in algebraic $K$-theory},
    conference={
    title={Applications of algebraic $K$-theory to algebraic
    geometry and
    number theory, Part I, II},
    address={Boulder, Colo.},
    date={1983},
    },
    book={
    series={Contemp.
    Math.},
    volume={55},
    publisher={Amer.
    Math.
    Soc.,
    Providence,
    RI},
    },
    date={1986},
    pages={389--406},
%     review={\MR{862644}},
%     doi={10.1090/conm/055.1/862644},
}

% \bib{thomason-trobaugh}{article}{
%     author={Thomason, R. W.},
%     author={Trobaugh, Thomas},
%     title={Higher algebraic $K$-theory of schemes and of derived
%     categories},
%     conference={
%     title={The Grothendieck Festschrift, Vol.\ III},
%     },
%     book={
%     series={Progr. Math.},
%     volume={88},
%     publisher={Birkh\"auser
%     Boston, Boston, MA},
%     },
%     date={1990},
%     pages={247--435},
%     %   review={\MR{1106918}},
%     %      doi={10.1007/978-0-8176-4576-2_10},
% }

% \bib{thomason-classification}{article}{
%     author={Thomason, R. W.},
%     title={The classification of triangulated subcategories},
%     journal={Compositio Math.},
%     volume={105},
%     date={1997},
%     number={1},
%     pages={1--27},
%     issn={0010-437X},
% %     review={\MR{1436741}},
% %     doi={10.1023/A:1017932514274},
% }

% % 
% % \bib{totaro}{article}{
% %     author={Totaro, Burt},
% %     title={The Chow ring of a classifying space},
% %     conference={
% %     title={Algebraic $K$-theory},
% %     address={Seattle, WA},
% %     date={1997},
% %     },
% %     book={
% %     series={Proc. Sympos. Pure Math.},
% %     volume={67},
% %     publisher={Amer. Math. Soc.},
% %     place={Providence, RI},
% %     },
% %     date={1999},
% %     pages={249--281},
% % %     review={\MR{1743244 (2001f:14011)}},
% % }
% % 
% % \bib{totaro-torsion}{article}{
% %     author={Totaro, Burt},
% %     title={Torsion algebraic cycles and complex cobordism},
% %     journal={J. Amer. Math. Soc.},
% %     volume={10},
% %     date={1997},
% %     number={2},
% %     pages={467--493},
% %     issn={0894-0347},
% % %     review={\MR{1423033 (98a:14012)}},
% % %     doi={10.1090/S0894-0347-97-00232-4},
% % }
% % 
% % \bib{totaro-injectivity}{article}{
% %     author={Totaro, Burt},
% %     title={Non-injectivity of the map from the Witt group of a variety to the
% %     Witt group of its function field},
% %     journal={J. Inst. Math. Jussieu},
% %     volume={2},
% %     date={2003},
% %     number={3},
% %     pages={483--493},
% %     issn={1474-7480},
% % %     review={\MR{1990223 (2004f:19010)}},
% % %     doi={10.1017/S1474748003000136},
% % }

% \bib{verdier}{article}{
%     author={Verdier, Jean-Louis},
%     title={Des cat\'egories d\'eriv\'ees des cat\'egories ab\'eliennes},
%     note={With a preface by Luc Illusie;
%     Edited and with a note by Georges Maltsiniotis},
%     journal={Ast\'erisque},
%     number={239},
%     date={1996},
%     pages={xii+253 pp. (1997)},
%     issn={0303-1179},
% %     review={\MR{1453167}},
% }

% % 
% % \bib{vezzosi}{article}{
% %    author={Vezzosi, Gabriele},
% %    title={On the Chow ring of the classifying stack of ${\rm PGL}_{3,\mathbf{C}}$},
% %    journal={J. Reine Angew. Math.},
% %    volume={523},
% %    date={2000},
% %    pages={1--54},
% %    issn={0075-4102},
% % %    doi={10.1515/crll.2000.048},
% % }
% % 
% % \bib{vistoli}{article}{
% %     author={Vistoli, Angelo},
% %     title={On the cohomology and the Chow ring of the classifying space of ${\rm PGL}_p$},
% %     journal={J. Reine Angew. Math.},
% %     volume={610},
% %     date={2007},
% %     pages={181--227},
% %     issn={0075-4102},
% % }

% \bib{waldhausen-free}{article}{
%     author={Waldhausen, Friedhelm},
%     title={Algebraic $K$-theory of generalized free products. I, II},
%     journal={Ann. of Math. (2)},
%     volume={108},
%     date={1978},
%     number={1},
%     pages={135--204},
%     issn={0003-486X},
% %     review={\MR{0498807}},
% }

% \bib{waldhausen-whitehead}{article}{
%     author={Waldhausen, Friedhelm},
%     title={Whitehead groups of generalized free products},
%     conference={
%     title={Algebraic K-theory, II: `Classical'' algebraic
%     K-theory and
%     connections with arithmetic},
%     address={Proc. Conf., Battelle Memorial
%     Inst., Seattle, Wash.},
%     date={1972},
%     },
%     book={
%     publisher={Springer,
%     Berlin},
%     },
%     date={1973},
%     pages={155--179.
%     Lecture
%     Notes
%     in
%     Math.,
%     Vol.
%     342},
% %     review={\MR{0370576}},
% }

\bib{weibel-kh}{article}{
author={Weibel, Charles A.},
title={Homotopy algebraic $K$-theory},
conference={
title={Algebraic $K$-theory and algebraic number theory
(Honolulu, HI,
1987)},
},
book={
series={Contemp. Math.},
volume={83},
publisher={Amer.
Math. Soc.,
Providence,
RI},
},
date={1989},
pages={461--488},
% review={\MR{991991}},
% doi={10.1090/conm/083/991991},
}

% \bib{weibel}{book}{
%     author={Weibel, Charles A.},
%     title={An introduction to homological algebra},
%     series={Cambridge Studies in Advanced Mathematics},
%     volume={38},
%     publisher={Cambridge University Press, Cambridge},
%     date={1994},
%     pages={xiv+450},
%     isbn={0-521-43500-5},
%     isbn={0-521-55987-1},
% %     review={\MR{1269324 (95f:18001)}},
% %     doi={10.1017/CBO9781139644136},
% }

% \bib{weibel-isolated}{article}{
%     author={Weibel, Charles A.},
%     title={Negative $K$-theory of varieties with isolated singularities},
%     booktitle={Proceedings of the Luminy conference on algebraic
%     $K$-theory
%     (Luminy, 1983)},
%     journal={J. Pure Appl. Algebra},
%     volume={34},
%     date={1984},
%     number={2-3},
%     pages={331--342},
%     issn={0022-4049},
% %     review={\MR{772067}},
% %     doi={10.1016/0022-4049(84)90045-8},
% }

% \bib{weibel-kbook}{book}{
%     author={Weibel, Charles A.},
%     title={The $K$-book},
%     series={Graduate Studies in Mathematics},
%     volume={145},
%     note={An introduction to algebraic $K$-theory},
%     publisher={American Mathematical Society, Providence,
%     RI},
%     date={2013},
%     pages={xii+618},
% %     isbn={978-0-8218-9132-2},
% %     review={\MR{3076731}},
% }

% % 
% % % \bib{williams2013}{article}{
% % %   author={Williams, Ben},
% % %   title={The $\Gm$--equivariant motivic cohomology of Stiefel varieties},
% % %   journal={Algebraic \& Geometric Topology},
% % %   volume={13},
% % %   date={2013},
% % %   pages={747--793},
% % % }
% % 
% % % \bib{whitehead}{book}{
% % %    author={Whitehead, George W.},
% % %    title={Elements of homotopy theory},
% % %    series={Graduate Texts in Mathematics},
% % %    volume={61},
% % %    publisher={Springer-Verlag},
% % %    place={New York},
% % %    date={1978},
% % %    pages={xxi+744},
% % %    isbn={0-387-90336-4},
% % % %    review={\MR{516508 (80b:55001)}},
% % % }
% % 
% % % \bib{woodward}{article}{
% % %     author={Woodward, L. M.},
% % %     title={The classification of principal ${\rm PU}_{n}$-bundles over a $4$-complex},
% % %     journal={J. London Math. Soc. (2)},
% % %     volume={25},
% % %     date={1982},
% % %     number={3},
% % %     pages={513--524},
% % %     issn={0024-6107},
% % % }

\end{biblist}
\end{bibdiv}
\end{document}